\documentclass[12pt]{amsart}

\usepackage{amsmath,amssymb}

\input xy
\xyoption{all}

\usepackage{graphicx}
\usepackage{caption}
\usepackage{subcaption}
\usepackage{booktabs}

\theoremstyle{plain}
\newtheorem{prop}{Proposition}
\newtheorem{thm}[prop]{Theorem}
\newtheorem{coro}[prop]{Corollary}
\newtheorem{lemm}[prop]{Lemma}
\newtheorem{conj}[prop]{Conjecture}

\newtheorem*{thm-int}{Theorem}
\theoremstyle{remark}

\newtheorem{rema}[prop]{Remark}

\theoremstyle{definition}
\newtheorem{defn}[prop]{Definition}

\newtheorem{cond}[prop]{Condition}

\def\G{{\mathcal G}}   

\def\KK{\boldsymbol{K}}
\def\FF{\boldsymbol{F}}

\def\kk{\boldsymbol{k}}

\def\no{\noindent}
\def\rk{{\rm rk}}
\newcommand{\trdeg}{{\rm tr}\, {\rm deg}}

\newcommand{\Ker}{{\rm Ker}}

\def\lra{\longrightarrow}
\def\ra{\rightarrow}

\def\A{{\mathbb A}}

\def\F{{\mathbb F}}

\def\P{{\mathbb P}}
\def\Q{{\mathbb Q}}

\def\Z{{\mathbb Z}}

\def\N{{\mathbb N}}

\def\bP{{\mathbb P}}

\def\ma{{\mathfrak a}}

\def\mc{{\mathfrak c}}

\def\mo{{\mathfrak o}}

\def\mm{{\mathfrak m}}
\def\ml{{\mathfrak l}}
\def\mq{{\mathfrak q}}

\begin{document}
\title[Homomorphisms and algebraic dependence]{Homomorphisms of multiplicative groups of fields 
preserving algebraic dependence}

\author{Fedor Bogomolov}
\address{Courant Institute of Mathematical Sciences, N.Y.U. \\
 251 Mercer str. \\
 New York, NY 10012, U.S.A.}
\address{National Research University Higher School of Economics, 
Russian Federation \\
AG Laboratory, HSE \\ 
Usacheva str., Moscow, Russia, 117312}

\email{bogomolo@cims.nyu.edu}

\author{Marat Rovinsky}
\address{National Research University Higher School of Economics, 
Russian Federation \\
AG Laboratory, HSE \\ 
Usacheva str., Moscow, Russia, 117312}
\email{marat@mccme.ru}

\author{Yuri Tschinkel}
\address{Courant Institute of Mathematical Sciences, N.Y.U. \\
 251 Mercer str. \\
 New York, NY 10012, U.S.A.}
\address{Simons Foundation\\
160 Fifth Av. \\
New York, NY 10010, U.S.A.}
\email{tschinkel@cims.nyu.edu}        


\begin{abstract}
We study homomorphisms of multiplicative groups of fields 
preserving algebraic dependence and show that such 
homomorphisms give rise to valuations. 
\end{abstract}

\maketitle

\setcounter{section}{0}       
\section*{Introduction}
\label{sect:introduction}

In this paper we formulate and prove a version of the 
Grothendieck section conjecture. For function fields of algebraic varieties over
algebraically closed ground fields this conjecture
states, roughly, that existence of group-theoretic sections of 
homomorphisms of their absolute Galois groups 
implies existence of geometric sections of morphisms 
of models of these fields. 

In detail, let $k$ be an algebraically closed 
field and $K=k(X)$ the function field of an algebraic variety  
$X$ over $k$. Let $G_K$ be the absolute Galois group of $K$.
Fix a prime $\ell$ not equal to the characteristic of $k$ and let $\G_K$ be the 
maximal pro-$\ell$-quotient of $G_K$, the Galois group of the maximal $\ell$-extension of $K$. 
Write 
$$
\G^a_K=\G_K/[\G_K,\G_K] \quad \text{ and } \quad \G^c_K:=\G_K/[\G_K,[\G_K,\G_K]],
$$ 
for the abelianization and its canonical central extension:
\begin{equation}
\label{eqn:central}
1\ra \mathcal Z_K\ra \G^c_K\stackrel{\pi_a}{\lra} \G^a_K\ra 1.
\end{equation}
Let $\Sigma_K=\Sigma(\G^c_K)$ be the set of topologically noncyclic 
subgroups $\sigma\subset \G^a_K$ whose 
preimages $\pi_a^{-1}(\sigma)\subset \G^c_K$ are abelian. 
It is known that function fields $K=k(X)$ of transcendence degree $\ge 2$ 
over $k=\bar{\F}_p$ are determined, 
modulo purely inseparable extensions, by
the pair $(\G^a_K, \Sigma_K)$ 
\cite{bt0}, \cite{bt1}, and \cite{pop}.

This raises the question of {\em 
functoriality}, i.e., the 
reconstruction of rational morphisms between algebraic varieties from 
continuous homomorphisms of absolute Galois groups of their function fields. This general fundamental 
question was proposed by Grothendieck and lies at the core of the Anabelian Geometry Program. 

The main open problem in this program 
relates to a Galois-theoretic criterium 
for the existence of rational sections of fibrations.
Let  
$$
\pi:X\to Y,
$$
be a fibration with connected generic fiber of dimension at 
least 1 over a base $Y$ of dimension $\ge 2$. 
This defines a field embedding 
$$
\pi^{\ast}:k(Y)\hookrightarrow k(X),
$$ 
with image of $L:=k(Y)$ algebraically 
closed in $K:=k(X)$. Dually, we have 
a surjective homomorphism of absolute Galois groups 
(restriction map) 
$$
G_{K}\to G_{L},
$$ 
as well as induced homomorphisms
$$
\G^c_{K}\to\G^c_{L}, \quad  
\G^a_{K}\to\G^a_{L}.
$$ 
A minimalistic version of Grothendieck's {\em Section conjecture} would be:

\begin{conj}
\label{conj:2}
Assume that $\pi_a:\G^a_{K}\to \G^a_{L}$ admits a section 
\begin{equation}
\label{eqn:sa}
\xi_a:\G^a_{L}\to \G^a_{K}
\end{equation}
such that
\begin{equation}
\label{eqn:cond}
\xi_a(\Sigma_{L})\subset\Sigma_{K}.
\end{equation}
Then there exist a finite 
purely inseparable extension 
$$
\iota^*:L\hookrightarrow L'=k(Y')
$$ 
and a rational 
map 
$$
\xi:Y'\to X,
$$ 
such that 
$$
\xi^*\circ\pi^*(L) = \iota^*(L)\subset L'.
$$
Thus $\xi(Y')$ is a section over $Y$, modulo purely inseparable extensions.
\end{conj}

Conjecture~\ref{conj:2} is closely 
related to questions considered in this note. 
Recall that, by Kummer theory, 
$$
\G^a_{K}=\mathrm{Hom}(K^{\times},\mathbb Z_{\ell}(1)),
$$ 
and that \eqref{eqn:sa} induces the 
dual homomorphism of (profinite) pro-$\ell$-completions of multiplicative groups
$$
\hat{\psi}:\hat{K}^{\times}\to\hat{L}^{\times}.
$$ 
Then \eqref{eqn:cond}
says 
that $\hat{\psi}$ respects the skew-symmetric pairings on 
$\hat{K}^{\times}$ and $\hat{L}^{\times}$, with values 
in the second Galois cohomology group of the corresponding field. 
The groups $\hat{K}^{\times}$ and 
$\hat{L}^{\times}$ contain $K^{\times}/k^{\times}$ and 
$L^{\times}/k^{\times}$, respectively. If 
the restriction $\psi$ of 
$\hat{\psi}$ to $K^\times/k^\times$ is ``rational'', i.e., 
$$
\psi:K^{\times}/k^{\times}\subseteq L^{\times}/k^{\times} \subset \hat{L}^\times, 
$$ 
then $\psi$ respects algebraic dependence, 
mapping algebraically dependent elements in $K^{\times}$ 
to algebraically dependent elements of $L^{\times}$ (modulo $k^\times$). 
For function fields this is 
equivalent to \eqref{eqn:cond}. 
This relates the ``minimalistic'' version of the Section conjecture for 
``rational'' maps to our main result, which we now explain.

Let $\nu$ be a nonarchimedean valuation of $K$, i.e., a 
homomorphism 
$$
\nu: K^\times \ra\Gamma_{\nu}
$$
onto a totally ordered group such that the induced map
$$
\nu: K \ra \Gamma_{\nu}\cup \{ \infty\}, \quad \nu(0)=\infty,
$$
satisfies a nonarchimedean triangle inequality. Let 
$$
\mm_{K,\nu}\subset \mo_{K,\nu}, \quad \KK_{\nu}:=\mo_{K,\nu}/\mm_{K,\nu},
$$
be the maximal ideal, valuation ring, and residue field 
with respect to $\nu$, respectively. If $K|k$ is a field extension 
and $\nu$ a valuation of $K$, then its restriction to $k$ 
is also a valuation; we have
$$
\mo_{K,\nu}^\times \cap k^\times =\mo_{k,\nu}^\times, \quad 
\mo_{K,\nu}^\times/ \mo_{k,\nu}^\times \subseteq K^\times/k^\times,
$$
and a natural surjection
$$
\mo_{K,\nu}^\times/ \mo_{k,\nu}^\times \longrightarrow\hspace{-3mm}\to 
\KK_{\nu}^{\times}/\kk_{\nu}^{\times}.
$$
We consider extensions of fields
$$
k\subseteq \tilde{k} \subseteq \tilde{k}_a\subset K,
$$
where $k$ is the prime subfield of $K$, i.e., 
$k=\mathbb F_p$ or $\mathbb Q$, and $\tilde{k}_a\subset K$ 
the algebraic closure of $\tilde{k}$ in $K$, i.e., the set of 
all algebraic elements over $k$ contained in $K$.
Assume that $\bar{x}_1, \bar{x}_2\in K^\times/k^\times$ 
safisfy 
\begin{equation}
\label{eqn:dependent}
\trdeg_{\tilde{k}}(\tilde{k}(x_1,x_2))\le 1, 
\end{equation}
for their lifts $x_1,x_2\in K^\times$; 
this does not depend on the choice of lifts. We write $x_1\sim_{\tilde k} x_2$ 
and say that $x_1$ and $x_2$ are contained 
in the same one-dimensional field; clearly
$1\sim_{\tilde{k}} x$, for all $\bar{x}\in K^\times/k^\times$. 
From now on, we 
use the same notation for an element $x\in K^\times$ and its image in $K^\times/k^\times$. 
Let 
$$
l\subseteq \tilde{l}\subseteq \tilde{l}_a\subset L
$$ 
be field extensions, 
where $l$ is the prime subfield of $L$,  
and let
$$
\psi:  K^\times/k^\times \ra L^\times /\tilde{l}^\times
$$
be
a homomorphism of multiplicative groups. 
We say that $\psi$ preserves algebraic 
dependence with respect to 
$\tilde{k}, \tilde{l}$  if
$$
x_1\sim_{\tilde{k}} x_2 \quad \Rightarrow \quad
\psi(x_1)\sim_{\tilde{l}} \psi(x_2), \quad \forall\, x_1,x_2\in K^\times/k^\times. 
$$

\begin{thm} 
\label{principal_theorem} 
Let $k\subseteq \tilde{k}\subset K$ and $l\subseteq \tilde{l}\subset L$ 
be field extensions as above. 
Assume that $\tilde{l}=\tilde{l}_a$ and that  
there exists a homomorphism
\begin{equation}
\label{eqn:psii}
\psi:  K^\times/k^\times \ra L^\times /\tilde{l}^\times,
\end{equation}
such that 
\begin{itemize}
\item 
$\psi$ preserves algebraic 
dependence with respect to $\tilde{k}$ and $\tilde{l}$;
\item 
there exist
$$
y_1,y_2\in \psi(K^\times/k^\times), \quad \text{ such that } \quad 
y_1\not\sim_{\tilde{l}} y_2;
$$
\item $\psi$ satisfies Assumption (AD) of Section~\ref{sect:proof}.  
\end{itemize}
Then either
\begin{itemize}
\item[(P)] there exists a field $F\subset K$
such that $\psi$ factors through
$$
K^{\times}/k^{\times}\longrightarrow\hspace{-3mm}\to 
K^{\times}/F^{\times},
$$
\item[(V)]
there exists a nontrivial valuation
$\nu$ on $K$
such that the restriction of $\psi$ to 
$$
\mo_{K,\nu}^{\times}/\mo_{k,\nu}^{\times}\subseteq K^\times/k^\times
$$ 
is trivial on 
$$
(1+\mm_{\nu})^\times/\mo_{k,\nu}^{\times} 
$$
and it factors through the reduction map 
$$\mo_{K,\nu}^{\times}/\mo_{k,\nu}^{\times}\longrightarrow\hspace{-3mm}\to 
\KK_{\nu}^{\times}/\kk_{\nu}^{\times}\ra L^\times/\tilde{l}^\times,
$$ 
\item[(VP)]
there exist a nontrivial valuation
$\nu$ on $K$ and a field $\FF_{\nu}\subset \KK_{\nu}$ such that 
the restriction of $\psi$ to $\mo_{K,\nu}^{\times}/\mo_{k,\nu}^{\times}$ 
factors through 
$$\mo_{K,\nu}^{\times}/\mo_{k,\nu}^{\times}\longrightarrow\hspace{-3mm}\to 
\KK_{\nu}^{\times}/\FF_{\nu}^{\times}\ra L^\times/\tilde{l}^\times.
$$ 
\end{itemize}
\end{thm}

In the geometric setting treated in \cite{GTPG}, 
when $K=\tilde{k}(X)$ is a function field of an algebraic variety 
$X$ over $\tilde{k}=\bar{\F}_p$, 
the center of the valuation $\nu$ arising in case (V) is, birationally,  
the image of the section, and the above theorem can be viewed as a 
``rational'' version of the minimalistic section 
conjecture. 
Here we extend the argument in \cite{GTPG} from function fields 
to arbitrary fields, under the additional technical assumption (AD) on $\psi$, 
which holds for $K$ of positive characteristic.

The idea of proof is to reduce the problem 
to a question in plane projective geometry 
over the prime subfield $k$. We 
view $\P(K):=K^{\times}/k^\times$ 
as a projective space over $k$. 
To establish Theorem ~\ref{principal_theorem} it suffices to 
show the existence of a subgroup 
$\mathfrak U\subset K^\times/k^\times$ 
such that:

\begin{cond}
\label{cond:flag}
For every projective line $\ml\subset \P(K)$, 
$\mathfrak U\cap \ml$
is either
\begin{itemize}
\item[(1)] the line $\ml$,
\item[(2)] a point $\mq\in \ml$, 
\item[(3)] the affine line $\ml\setminus \mq$, or
\item[(4)] if $k=\mathbb Q$, a set projectively equivalent to 
$$
\mathbb Z_{(p)}\subset \mathbb A^1(\mathbb Q)\subset \mathbb P^1(\Q),
$$ 
the set of rational numbers with denominator coprime to $p$. 
\end{itemize}
\end{cond}

Indeed, such a subgroup is necessarily either $F^\times/k^\times$ for some subfield 
$F\subset K$, or
$\mo_{K,\nu}$, for some valuation $\nu$ (see Section~\ref{sect:proof}).  
By construction, the homomorphism $\psi$
will satisfy the cases 
(P) or (V) in Theorem~\ref{principal_theorem}, respectively.

To find such $\mathfrak U$, we 
use results of \cite{HalesStraus} and \cite{AbSubgrps}. 
First we deduce that
the restriction of $\psi$ to 
every plane $\mathbb P^2\subset \P(K)$ is 
either an embedding or is induced by a natural construction from 
some nonarchimedean valuation (see Section~\ref{sect:lines}). 
We distinquish two cases:
\begin{itemize}
\item there exists a line $\ml\subset \P(K)$ such that 
the restriction of $\psi$ to $\ml$ is injective,
\item no such lines exist.
\end{itemize}
In the first case, property $(4)$ of Condition~\ref{cond:flag} 
does not occur, and the proof works uniformly for 
$k=\mathbb F_p$ or $\mathbb Q$. In the second case, the proofs are
slightly different, leading to 
a case-by-case analysis in Section~\ref{sect:lines}. 

\

\no
{\bf Acknowledgments.} 
This work was partially supported by 
Laboratory of Mirror Symmetry NRU HSE, RF grant 14.641.31.0001.
The first and second authors were funded by
the Russian Academic Excellence Project `5-100'.
The first author was also supported by a Simons Fellowship 
and by the EPSRC program grant EP/M024830. 
The third author was partially supported by the NSF grant  1601912.

\section{Projective geometry}
\label{sect:prelim}

Let $\P$ be a projective space over a field $k$ and 
$\Pi(\mq_0,\dots,\mq_n)\subseteq \P$ the 
projective envelope of points $\mq_0,\dots,\mq_n\in \P$. 
Working with lines and planes, we write 
$$
\ml=\ml(\mq_0,\mq_1),\quad 
\text{resp.} \quad \Pi=\Pi(\mq_0, \mq_1,\ma_2),
$$
for a projective line through $\mq_0,\mq_1$, or  
a plane through $\mq_0,\mq_1,\mq_2$.

Let $\nu$ a nonarchimedean valuation of $k$, 
$\mo=\mo_{\nu}$ the corresponding valuation ring, 
and $\kk_{\nu}$ the residue field.  
Fixing a lattice 
$$
\Lambda\simeq \mo^{n+1}\hookrightarrow k^{n+1},
$$
we obtain a natural surjection
\begin{equation}
\label{eqn:rho}
\rho=\rho_{\Lambda} : \P^n(k)\ra \P^n(\kk_{\nu}).
\end{equation}
A {\em 3-coloring} of $\P^2(k)$ is 
a surjection
\begin{equation}
\label{eqn:kappa}
\mc: \P^2(k)\ra \{ \bullet, \circ,\star\},
\end{equation}
onto a set of 3 elements,
such that 
\begin{itemize}
\item every $\ml\subset \P^2(k)$ is colored in exactly two colors, i.e., $\mc(\ml)$ consists of two elements.
\end{itemize}
A 3-coloring is called {\em trivial of type} 
\begin{itemize}
\item {\em I}: if there exists a line $\ml\subset \P^2$ such that $\mc$ is constant on 
$\P^2\setminus \ml$,
\item {\em II}: if there exists a point $\mq\in \P^2(k)$ such that for every $\ml\subset \bP^2$ containing $\mq$, $\mc$ is constant on $\ml\setminus \mq$. 
\end{itemize}
It was discovered early on, that such colorings are related to valuations, see, e.g., 
\cite{HalesStraus}. The same structure resurfaced in the study of {\em commuting} elements of Galois groups 
of function fields in \cite{CommElts}, exhibiting unexpected projective 
structures {\em within} $\G^a_K$. 
This was a crucial step in the recognition of inertia and decomposition 
subgroups in $\G^a_K$.

Precisely, we have (see \cite[Theorem 2]{HalesStraus} and  \cite{CommElts}):

\begin{prop}
\label{prop:plane} 
Assume that $\P^2(k)$ carries a 3-coloring. Then there exists a nonarchimedean valuation $\nu$
such that the coloring $\mc$ in \eqref{eqn:kappa} is induced from a trivial covering
$$
\mc_{\nu} : \P^2(\kk_{\nu})\ra  \{ \bullet, \circ,\star\}, 
$$
for some $\rho$ as in \eqref{eqn:rho}. 
\end{prop}

\section{Flag maps}
\label{sect:flag}

We will consider maps (respectively, homomorphisms)  
$$
f:\P\ra A
$$
from projective spaces over $k$ 
to a set (respectively, an abelian group). 
The map $f$ is called a {\em flag map} if its 
restriction $f_{\Pi}$ to every finite dimensional projective subspace $\Pi\subset \P$ is a flag map. 
For $k=\mathbb F_p$ and 
$$
f:\P^n(\F_p)\ra A,
$$
this means that there exists a flag of projective subspaces
\begin{equation}
\label{eqn:flag}
\P^n\supset \P^{n-1}\ldots \supset \P^1\supset \P^0=\mathfrak q
\end{equation}
such that $f$ is constant on $\P^i(\F_p)\setminus \P^{i-1}(\F_p)$, for all $i=1, \ldots, n$. 
For $k=\Q$ and 
$$
f:\P^n(\Q)\ra A,
$$
this means that either
\begin{itemize}
\item there is a flag as in \eqref{eqn:flag} so that $f$ is constant on 
 $\P^i(\Q)\setminus \P^{i-1}(\Q)$, for all $i=1, \ldots, n$, or
\item there exist a prime $p$, a surjection 
$$
\rho=\rho_{\Lambda} : \P^n(\Q)\ra  \P^n(\F_p)
$$
as in \eqref{eqn:rho}, and a flag map 
$$
\bar{f}: \P^n(\F_p)\ra A,
$$
such that 
$$
f= \bar{f}\circ \rho.
$$
\end{itemize}

\begin{prop} \cite{CommElts}
\label{prop:main}
Let 
$$
f:\P(K)=K^\times/k^\times\ra A
$$
be a group homomorphism which is also a flag map. 
Then there exist a valuation $\nu$ of $K$ 
and a homomorphism $r:\Gamma_{\nu}\ra A$ 
such that $f$ factors through
$$
K^\times/k^\times \stackrel{\nu}{\lra} \Gamma_{\nu} \stackrel{r}{\lra}  A.
$$
\end{prop}
A flag map $f$ on $\P^n(k)$ 
defines a map 
\begin{equation}
\label{eqn:rst}
\begin{array}{rcl}
\hat f:  \hat{\P}^n & \ra &  A \\
  \lambda & \mapsto & f_{\rm gen}(\lambda) 
\end{array}
\end{equation}
on the dual space, by assigning to a 
projective hyperplane the {\em generic} value of $f$ on this hyperplane. 
The following lemma generalizes results in \cite[Section 2]{CommElts}. 

\begin{lemm}
\label{lemm:restrict}
A map $f:\P\ra A$ is a flag map if and only if for every $\Pi=\P^2(k)\subset \P$
the restriction $f_{\Pi}$ is a flag map.
\end{lemm}

\begin{proof}
Assume the claim for every $\P^{n-1}\subset \P^n$, for $n\ge 3$.

\

{\em Step 1.}
Consider $\mq_1,\ldots, \mq_r\in \P^n(k)$ so that 
$f(\mq_i)$ are not generic in $\P^n$ and let 
$\Pi=\Pi(\mq_1,\ldots, \mq_r)\subset \P^n(k)$ 
be their projective envelope.
If $\dim(\Pi) \geq n-1$ then, for some $i$, 
$$
f(\mq_i)=f_{\rm gen}(\Pi).
$$ 
Indeed, if $\mq_1,\ldots,\mq_r$ are nongeneric then
any subset generates a subspace of dimension
$\leq n-2$ and hence $\{ \mq_1, \ldots, \mq_r\}\subset \Pi$, $\dim(\Pi)\leq n-2$, 
by induction. Then $f$ is constant outside of $\Pi$ and hence, by induction, 
a flag map.
Thus we may assume $\dim(\Pi) < n-1$.

\

{\em Step 2.}
$\hat{f}$ takes at most two values. Indeed, assume it takes 
distinct values $a_i$ on hyperplanes $\P^{n-1}_i$, $i=1,2,3$, 
so that $f$ is constant on affine subspaces  $\A^{n-1}_i\subset \P^{n-1}_i$. 
There exists a $\mq_3\in \A_3^{n-1}$, not contained
in $\P^{n-1}_1\cup \P^{n-1}_2$, since 
$\A_1^{n-1}, \A_2^{n-1},\A^{n-1}_3$ are disjoint
and
$$
(\P^{n-1}_1\cup \P^{n-1}_2)\setminus (\A_1^{n-1}\cup \A_2^{n-1})= 
\P^{n-2} = \P^{n-1}_1\cap \P^{n-1}_2,
$$
which does not contain an affine space $\A^{n-1}$.
Consider the projection
$$
\pi_3:  \P^n \setminus \{ \mq_3 \}\to \P^{n-1}
$$
from $\mq_3$. Then there exists a point 
$$
\mq\subset \pi_3(\A_1^{n-1})\cap \pi_3(\A_2^{n-1})\subset \P^{n-1}.
$$
The restriction of $f$ to $\ml(\mq_3,\mq)$ takes
three values, and $f$ is not a flag map on $\P^2$ containing this line, 
contradiction. 

\

{\em Step 3.}
Consider $\P^{n-1}_1$ with generic value $a_1$ and let $\mq\notin \P^{n-1}_1$ be a
point $f(\mq)\neq a_2$. 
Let $\P'\subset \P^{n-1}_1$
be a maximal projective subspace with a generic value different from $a_1$.
On any $\P^{n-1}_{\mq}\subset \P^n$ containing $\mq$ and
such that $\P^{n-1}_{\mq}\cap \P^{n-1}_1\neq \P'$ the generic value $\hat f_{\mq}=a_1$.
Indeed, it is generated by $\mq$ and points of  $\P^{n-1}_1 \setminus \P'$.
Thus $\hat{f}_{\mq}\neq a_2$, and  since $f$ takes only two values, 
$\hat f_{\mq}=a_1$.
In particular, on $\hat{\P}^{n-1}_{\mq}\subset \hat{\P}^{n}$, 
which corresponds to all
hyperplanes containing $\mq$, there is at most one point
with $\hat f\neq 1$. This can only occur if $\dim(\P')= n-1$.

The same argument holds for $\P^{n-1}_2$ and the corresponding subspace
$\hat{\P}^{n-1}_u\subset \hat{\P}^{n}$.
Since $\hat{\P}^{n-1}_{\mq}$ and $\hat{\P}^{n-1}_u$ intersect in $\P^{n-2}$ we obtain
a contradiction, 
unless all the points with value $a_i$
are contained in a hyperplane (for $i=1$ or $2$), 
if $n > 2$. Thus $f$ is a flag map.
\end{proof}

The following lemma is in \cite[Lemma 5.2]{AbSubgrps} 
(corrected in \cite[Prop.~3.4.1]{CommElts}).

\begin{lemm}
\label{lemm:qq}
Let $k=\Q$ or $k=\F_p$, with $p > 2$, and let 
$$
f:\P^2(k)\ra A
$$ 
be such that for every line $\ml\subset \P^2$ the 
restriction $f_{\ml}$ is a
flag map. Then $f$ is a flag map. 
\end{lemm}

\begin{proof}
The proof of Lemma~\ref{lemm:restrict} works up
to Step 2. Thus $\hat f$ takes two values $a_1,a_2$
and either
\begin{enumerate}
\item there is exactly one point $\mq$ with $f(\mq)=a_3$ or
\item $\hat{f}$ takes two values.
\end{enumerate}
In the first case, we apply Proposition~\ref{prop:plane}.
In the second case, either
\begin{itemize}
\item one of the values is concentrated on a line or
\item there exist $x_1,x_2,x_3$ and $y_1,y_2,y_3$
spanning $\P^2(k)$ with $f(x_i)=a_1$ and $f(y_i)=a_2$.
\end{itemize}
Consider the projection
$$
h_{x_1} : \P^2\setminus x_1\to \ml:=\ml(x_2,x_3).
$$ 
The generic value on $\ml$ is $a_1$ 
and hence there is at most
one $\mq_2\in\ml$ with $f(\mq_2)=a_2$.
The generic value of $f$ on any $\ml(x_1,x), x\in \ml\setminus \mq_2$
is $a_1$. Over  $k=\F_p$, the number
of points with $f=a_2$ is $\leq 2p$.
The same argument applies for $y_1$, and it follows 
$$
\P^2(\F_p)=p^2+p+1 \leq 4 p, \quad \text{i.e., } p=2.
$$ 
This approach works also over $k=\mathbb Q$.
\end{proof}

\begin{rema} 
Lemma~\ref{lemm:qq}  fails
over $\mathbb F_2$, since any map with two values on
$\P^1(\mathbb F_2)$ is a flag map. 
\end{rema}

\begin{lemm} 
\label{flag_homomorphisms} 
Let 
$$
f:\P(K)=K^\times/k^\times\ra A
$$ 
be a group homomorphism whose restriction to every line is a flag map,
and such that there exists a 
plane $\Pi=\Pi(1,x,y)$, with $ f(x),f(y)\neq 1$, 
and $f_{\Pi}$ not a flag map.
Then 
$$
f(x)=f(y)\quad \text{ and } \quad f(x)^2=1.
$$
In particular, if $f$ is not a flag map,  
then $k=\mathbb F_2$ and $f^2$ is a flag map. 
\end{lemm}

\begin{proof} 
By Lemma~\ref{lemm:qq}, $k=\mathbb F_2$, 
so that $K^{\times}/k^{\times}=K^{\times}$. Then
$f_{\Pi}$ takes two values and  
is constant, with distinct values, on two 
triples of noncollinear points. 
Thus
\begin{itemize}
\item[(i)] 
$f$ is constant on  
$\mathfrak l\cup\{\mq\}$, where $\mathfrak l\subset \Pi$ is a line 
(containing the remaining seventh point) and 
$\mq\in \Pi\setminus\mathfrak l$, and 
\item[(ii)] 
$f$ is constant on $\{x_1,x_2,x_3\}:=\Pi\setminus(\mathfrak l\cup\{\mq\})$; 
put $a:=f(x_i)$.
\end{itemize}
After a shift, we may assume that $\mq=1$, so 
$f(\mathfrak l)=1$. Suppose that $a^2\neq 1$. Let $\P^3_i$  be the 
projective envelope of $\Pi\cup\{x_i^2\}$. We claim that $x_i^2\in \P^3_i$ 
is the only point with $f(x_i^2)=a^2$, in particular, 
$f$ is constant, on the complement to $x_i^2$, 
on lines in $\P^3_i$ passing through $x_i^2$. 

Note that $f$ takes three values on
$\P_i^2:=\Pi(1,x_i,x_i^2)$, and thus is a flag map on $\P_i^2$. 
Let $y_i=x_i+1$ be the only point of 
$\ml(1,x_i)\setminus\{1,x_i\}$. Then $f(y_i)=1$, 
since $y_i\in\mathfrak l$. Thus, $f(1)=f(y_i)=f(y_i^2)=1$ 
and $f(x_i)=f(x_iy_i)=a$ and $f(x_i^2)=a^2$, so $1$ is the $f$-generic 
value on $\P_i^2$, and therefore, the remaining point 
$$
\mq_i=x_i^2+x_i+1\in 
\P_i^2\setminus \left(\ml(1,x_i)\cup\ml(1,x_i^2)\cup \ml(x_i,x_i^2)\right)
$$
is also $f$-generic, i.e., $f(\mq_i)=1$. 

We have $4=2^3-2^2$ points in $\P^3_i\setminus(\Pi\cup \P_i^2)$, where 
$f$ is not yet determined. They are contained, for $r\neq i$, 
in $\Pi(1,x_i,x_i^2+x_r)\setminus\ml(1,x_i)$, 
which consists of 
$$
x_i^2+x_r, \quad x_i^2+x_j, \quad x_i^2+x_j+1, \quad x_i^2+x_r+1.
$$ 
Each of these is contained in a line intersecting 
$\Pi$ and $\P_i^2$ at points with different values $a$ 
and $1$. Indeed, for $\{i,j,r\}=\{1,2,3\}$, one has 
\begin{itemize} 
\item $x_i^2+x_r\in\ml(q_i,x_j)\cap\ml(x_i^2,x_r)$, 
so  
$$
f(x_i^2+x_r)\in\{1,a\}\cap\{a^2,a\}=\{a\};
$$  
\item $x_i^2+y_r\in\ml(x_i^2+1,x_r)\cap\ml(x_i^2,y_r)$, 
and  
$$
f(x_i^2+y_r)\in\{1,a\}\cap\{a^2,1\}=\{1\}.
$$ 
\end{itemize} 
Note that $f$ takes three 
values on $\ml(x_1^2,x_2^2+x_1)$: 
$$
f(x_1^2)=a^2, \quad f(x_2^2+x_1)=a, \quad  f(x_3^2+x_1+1)=1.
$$
If $a^2\neq 1$ we get a contradiction to our
assumption that $f$ takes only two values and is flag on any 
line in $\Pi(x_1,x_2,x_3)$.
Thus on every $\Pi(1,x,y)$ where $f$ is not a flag map, 
$f^2\equiv 1$,  
hence is a flag map. 
\end{proof}

\begin{rema}
Under conditions of Lemma~\ref{lemm:qq}, 
$f^2$ is always a flag map. 
In particular,  
if $A$ has no $2$-torsion, then Proposition~\ref{prop:main}
holds as well.
\end{rema}

\begin{lemm}
\label{2-maps}
Assume the conditions of Lemma~\ref{lemm:qq} and 
that the two-torsion part $A[2]\subseteq A$ is nontrivial. Consider
the composition
$$
f_2: \P(K) \stackrel{f}{\lra} A \stackrel{r_2}{\longrightarrow} A/A[2],
$$ 
with $r_2$ the projection.
Then $f_2$ is a flag map on every plane $\Pi\subset \P(K)$.
\end{lemm}

\begin{proof}
If $f$ is a flag map on $\Pi(1,x,y)$ then $f_2$ is also flag.
If $f$ is not a flag map, then we apply Lemma~\ref{flag_homomorphisms}: 
$f$ takes only two values $0$ or $1$, and $f(x)^2=1$,
thus $f(x) =1$.

In particular, $f_2 \equiv 1$ on $\Pi(1,x,y)$ and hence is
a flag map. Thus $f_2$ is a flag map on every plane and hence a flag map.
\end{proof}

To summarize, if $A\neq A[2]$ then $f$ determines a valuation $\nu$.
If $A=A[2]$, then $f$ is trivial on some
subfield $K'\subset K$ such that $K|K'$
is a  purely inseparable extension of exponent
$2$.


\section{$\Z_{(p)}$-lattices}
\label{sect:zp}

Let $p$ be a prime number and $\Z_{(p)}\subset \Q$ the set of rational numbers
with denominator coprime to $p$.
A $\Z_{(p)}$-lattice, or simply, a lattice $B\subset \Q^{n+1}$ is 
a $\Z_{(p)}$-submodule such that 
$B\otimes_{\Z_{(p)}} \Q=\Q^{n+1}$. 
Given a lattice $B\subset \Q^{n+1}$ and an element 
$x\in \Q^{n+1}\setminus 0 $ there
exists an element $x_B\in B\setminus pB$ such that 
$x$ and $x_B$ define the same point in $\P^n(\Q)$,  
this element is unique in $B\setminus pB$, 
modulo scalar multiplication by $\Z_{(p)}^\times$. 
Lattices $B,B'\subset \Q^{n+1}$ 
are called equivalent if $B= a \cdot B'$, for some $a\in \Q^\times$.

In this section, we consider maps
$$
f: (\Q^{n+1}\setminus 0) \ra A
$$
which are invariant under scalar multiplication  by $\Q^\times$; we use the same 
notation for the induced map
$$
f:\P^n(\Q)\ra A.
$$
We say that $f$
is {\em induced from $\P^n(\Z/p)$ via a lattice $B$} 
if there exists a map 
$$
\bar{f}:\P^n(\Z/p)\ra A
$$ 
such that
$$
f(x)=(\bar{f} \circ \rho_B)(x_B), \quad \text{ for all } \quad x\in \P^n(\Q),
$$  
where 
$$
\rho_B : (B\setminus  pB) \to (B/pB) \setminus 0\to \P^n(\Z/p).
$$
This is well-defined since $\rho_B$ is invariant under $\Z_{(p)}^\times$. 
Such lattices will be called
$f$-compatible, or simply compatible. 
If $f$ is induced from $\P^n(\Z/p)$ via a lattice $B$ 
then it is also induced via any equivalent lattice. 

\begin{lemm}
\label{lemm:drr}
Assume that $f$ is induced from $\P^n(\Z/p)$
via a lattice $B$. Let $\bar{x}\in \P^n(\Z/p)$ and choose an 
$x_B\in B\setminus pB$ 
such that $\rho_B(x_B)=\bar{x}$. Let
$$
N_{\bar x, B}:= \{ x\in \Q^{n+1}\setminus 0 \,|\,  
x = x_B\in \P^n(\Q)\}. 
$$
If $B'$ is another $f$-compatible lattice
such that 
$$
N_{\bar x', B'} = N_{\bar x, B}, \quad \text{ for some } \quad \bar{x}'\in \P^n(\Z/p),
$$
then $B'$ is equivalent to $B$.
\end{lemm}

\begin{proof}
We have 
$$
N_{\bar x, B}=\Q^\times\cdot (x_B + pB), \quad x_B\in (B\setminus pB),  \quad 
\rho_B(x_B) = \bar{x}.
$$
Consider $z\in (\Q^{n+1}\setminus 0)\setminus N_{\bar x, B}$
such that $x_B +z$ projects to $\bar{x}\in \P^n(\Z/p)$.
Note that $z\not\in (B\setminus pB)$ 
since otherwise $\overline{x+ z}\neq \bar x$.  
Furthermore, $z\notin \Q^{n+1}\setminus B$, since
otherwise
$z\in p^{-m}(B \setminus p B)$, for some $m\in \N$, and 
$\overline{x_B+z} = \bar z\neq \bar x$.
Thus $z\in pB$, and the lattice $B_x$ spanned by such $z$ equals $pB$. 
Hence for any $y\in N_{\bar x,B}$,  
$B_y = p^m B$, for some $m\in \Z$. 
\end{proof}

In the discussion below, we use projective and affine geometry. The following lemma connects these concepts. 

\begin{lemm}
\label{lemm:aff-proj}
Let $L=\A^2(\Z_{(p)}) \subset \A^2(\Q)\subset \P^2(\Q)$ and 
$$
\rho_L: L\ra \bar{L}=\A^2(\Z/p)
$$ 
be the reduction map.
Then there exists a unique equivalence class of $\Z_{(p)}$-lattices $B\subset \Q^3$ such that $\rho$ 
extends to 
$$
\rho_B: B\setminus pB\ra \P^2(\Z/p)
$$ 
and such that the corresponding map 
$\P^2(\Q)\ra \P^2(\Z/p)$ coincides with $\rho_L$ upon restriction to $L$. 
\end{lemm} 

\begin{proof}
Let $\tilde{x}\in \Q^3$ be an element projecting to $x \in L\subset \P^2(\Q)$. Consider 
a subspace $\Q^2\subset \Q^3$ such that the corresponding line $\ml\subset \P^2$ is disjoint from $L$. Then the preimage of 
$L$ in $x+\Q^2\subset \Q^3$ coincides with the $x+B'$, for some lattice $B'\subset \Q^3$, and 
the lattice generated by $x$ and $pB'$ is the desired lattice $B$. Indeed, $x+pB'$ is the preimage of a point in $\bar{L}$, hence
the sublattice $B$ defines a map $\rho_B$ with desired properties; all such lattice are equivalent, by  Lemma~\ref{lemm:drr}. 
\end{proof}

\begin{lemm}
\label{indunique -from line}
Assume that $f:\P^1(\Q)\ra A$ is induced
from a nonconstant map $\bar{f}:\P^1(\Z/p)\ra A$, via some lattice.
\begin{enumerate}
\item If  $\bar{f}$ is flag map, then
there are exactly two equivalence classes of $f$-compatible
lattices $B_1, B_2\subset \Q^2$.
\item If $\bar f$ is not a flag map, then
there is exactly one equivalence class of 
$f$-compatible lattices $B\subset \Q^2$. 
\end{enumerate}
\end{lemm}

\begin{proof}
By assumption, $f$ is induced via some $\rho_B$. 
Fix generators $x,y\in B$ 
such that $f(y)\neq f(x)$, in particular 
$\rho_B(x_B)\neq \rho_B(y_B)\in \P^1(\Z/p)$. 
We have 
$$
f(y+ pB) =f (y) \text{ and } f(x+pB)=f(x)\neq f(y).
$$
Any lattice $B'\subset \Q^2$ is equivalent to a lattice
with $x$ as a generator. Since $B'/\Z_{(p)} \cdot x \simeq \Z_{(p)}$,  $B'$ is
one of the following:
$B_i:=\langle x, p^i y\rangle$, 
for some $i\in \Z$.
If $f$ is induced from $B_i$, for some $i < -1$, then  
$$
f(x+p(p^i y))=f(x)\neq f(y) \text{ and } 
f(x+p(p^i y))= f (p^{-i-1} x + y) = f(y),
$$
a contradiction. 
The same argument gives a contractions when $i > 1$.
Thus $i=1,0$, or $-1$.

Analysis of values of $\bar f$ at other points
of $\P^1(\Z/p)$ leads to further restrictions.
We have the following cases: 
\begin{enumerate}
\item $\bar f$ is constant on $\P^1(\Z/p)\setminus \rho_B(y_B)$.
\item $\bar f$ is not
constant on the complement to a point in $\P^1(\Z/p)$.
\end{enumerate}
   
In Case (1), $f(x+y) =f(x)$, excluding
$i= 1$.
Then we have exactly two lattices $B_0, B_{-1}$
such that $f$ is induced from these (or equivalent) lattices. 

In Case (2), if $f$ is induced from $B_{-1}$
then $f(\kappa x+y)= f(y)$, for any $\kappa\in \Q$, and hence
$\bar f$ is constant on $\P^1(\Z/p)\setminus \rho_B(x_B)$,
contradicting the second condition.
Thus there is only one compatible lattice $B_0= B$, modulo equivalence.
\end{proof}

\begin{lemm} 
\label{lemm:eight}
Assume that $f: \P^2(\Q) \ra A$ satisfies the following: 
\begin{enumerate}
\item $f$ takes three values; 
\item 
$f$ takes at most two values on every line $\ml\subset \P^2$;  
\item 
on every $\P^1(\Q)\subset \P^2(\Q)$, 
$f$
is induced from a flag map on $\P^1(\Z/p)$, via $\rho_{B'}$, 
for some lattice $B'\subset \Q^2$.
\end{enumerate}
Then there are exactly three equivalence classes
of lattices $B_i\subset \Q^3$
such that $f$ is induced from a
flag map $\bar f:\P^2(\Z/p)\ra A$, via $\rho_{B_i}$, $i=1,2,3$. 
\end{lemm}

\begin{proof}
Follows from Proposition~\ref{prop:plane}, applied to $k=\Q$  
(see also \cite{HalesStraus} 
or \cite{CommElts}). The first two conditions imply 
that there exists a lattice $B\subset \Q^3$
such that $f$ is induced from some 
map $\bar{f}: \P^2(\Z/p)\ra A$, via $B$. 
Applying both statements of Lemma~\ref{indunique -from line}, 
we conclude that $\bar{f}$ is a flag map, with 3 distinct values.  
Hence 
$$
\P^2(\Q)=S_1\sqcup S_2\sqcup S_3,
$$
with $S_1$ the preimage of an affine plane in $\P^2(\Z/p)$, 
$S_2$ an affine line, and $S_3$ a point in $\P^2(\Z/p)$, and $f$
is constant on these sets. 

By Lemma~\ref{indunique -from line},
for any $B'\subset \Q^3$ such that $f$ is induced
from $\P^2(\Z/p)$ via $\rho_{B'}$, 
the restriction of 
$f$ to any $(\Q^2\setminus 0)\subset (\Q^3\setminus 0)$ is induced
from a flag map on $\P^1(\Z/p)$. 

Thus $f$ is also induced from a flag map, via $\rho_{B'}$.
On the other hand, 
in coordinates $x_1,x_2,x_3$,
we have 
$$
S_1 = \{ x_1\neq 0\}, \quad S_2 = \{ x_1= 0, x_2\neq 0 \}, \quad S_3 = \{ x_1= x_2=0, x_3\neq 0 \},
$$
and the only possible coordinates compatible with the structures
on all $\P^1(\Q)$ are  
$$
x_1, \frac{x_2}{p}, \frac{x_3}{p},\quad 
x_1,x_2,\frac{x_3}{p}, \text{ and } \quad x_1,x_2,x_3. 
$$
This gives exactly three equivalence classes of 
$f$-compatible lattices.
\end{proof}

\begin{coro}
\label{coro:point}
Assume that we are in the situation of Lemma~\ref{lemm:eight} so that $f$ 
is induced from a flag map 
$$
\bar{f}:\P^2(\Z/p)\ra \{ 1,2,3\}.
$$ 
Let 
$\mq\subset \ml\subset \P^2(\Z/p)$ be the corresponding flag. 
Let $B'$ be another $f$-compatible lattice, 
$$
\bar{f}':\P^2(\Z/p)\ra \{1,2,3\}
$$ 
the corresponding flag map inducing $f$, 
and $\mq'\subset \ml'\subset \P^2(\Z/p)$
the associated flag. 
If $\bar{f}(\mq)= \bar{f}'(\mq')$ then $B$ and $B'$ are
equivalent and $\bar{f}=\bar{f}'$. 
\end{coro}

Thus if $f:\P^1(\Q)\ra A$ is induced from $\bar{f}:\P^1(\Z/p)\ra A$ then 
one of the following holds:
\begin{enumerate}
\item 
The map $\rho: \P^1(\Q)\to \P^1(\Z/p)$ inducing $f$ is unique;
in particular, for any $\mq\in \P^1(\Z/p)$ the subset
$L_p := \rho^{-1} (\P^1(\Z/p)\setminus \mq)$ is affinely isomorphic to $\Z_{(p)}\subset \Q\subset \P^1(\Q)$
and the map $L_p \to \Z/p$ is uniquely defined as a linear
map 
$$
\begin{array}{ccc}
\Z_{(p)} & \to &  \Z/p\\
x           &  \mapsto &  x \pmod{p}.
\end{array}
$$
\item $f$ is a flag map and there are two
possible $\rho: \P^1(\Q)\to \P^1(\Z/p)$ inducing $f$.
\end{enumerate}

\section{Basic field theory}
\label{sect:fields}

Let 
$$
k\subseteq \tilde{k}\subseteq \tilde{k}_a\subseteq K
$$ 
be an extension of fields. 
We say that $x_1, x_2 \in K^\times/k^\times$ 
are algebraically dependent
with respect to $\tilde{k}$ 
if they satisfy Equation~\eqref{eqn:dependent} from the Introduction; 
in this case, we write $x_1\sim_{\tilde{k}} x_2$, or simply $x_1\sim x_2$. 
We record the following obvious properties of this equivalence relation:
\begin{itemize}
\item[(AI)] 
If $x_1\sim_{\tilde{k}} x_2$,  
$x_1/x_2\notin \tilde{k}^\times_a/k^\times$, and 
$x\not\sim_{\tilde{k}} x_1$ then $x_1/x\not\sim_{\tilde{k}} x_2/x$. 
\item[(AG)] The set of nonconstant algebraically dependent elements, together 
with $(\tilde{k}^\times_a/k^\times)$ forms a subgroup of $K^\times/k^\times $.
\end{itemize}

\begin{lemm} 
\label{lemm:dontneed}
Let $K|k$ and $L|l$ be field extensions 
and 
\begin{equation}
\label{eqn:pssi}
\psi:K^{\times}/k^{\times}\to L^{\times}/l^{\times}
\end{equation}
a homomorphism such that its restriction to  
$\mo_{K,\nu}^\times/ \mo_{k,\nu}^\times$ factors as
\begin{equation}
\label{eqn:psi}
\mo_{K,\nu}^{\times}/\mo_{k,\nu}^{\times}\longrightarrow\hspace{-3mm}\to 
\KK_{\nu}^{\times}/\kk_{\nu}^{\times} \stackrel{\psi_\nu}{\lra} L^{\times}/l^{\times}.
\end{equation}
Assume that $\psi_{\nu}$ 
preserves algebraic dependence with respect to $\kk_{\nu}$ and $l$.
Then $\psi$ also preserves algebraic dependence with respect to $k$ and 
$l$.
\end{lemm} 

\begin{proof}
Let $k(x)\subset K$ be a purely transcendental extension 
and 
$$
E=\overline{k(x)}\subset K
$$ 
its algebraic closure in $K$. 
We claim that the restriction of $\psi$ to $E^\times/k^\times $ preserves 
algebraic dependence.
This is clear if $\psi$ 
is injective and preserves algebraic
dependence. Now assume that $\psi$ is defined through a 
valuation $\nu$, i.e., as in \eqref{eqn:psi}.
There are two cases:

\

\noindent
{\em Case 1.} $\nu(k^{\times})=\nu(E^{\times})$. 
Then 
$$
E^{\times}=\mo_{E,\nu}^{\times} \cdot k^{\times}.
$$ 
Since $\psi_{\nu}$ preserves algebraic dependence 
with respect to $\kk_{\nu}$ and $l$, the claim follows.   
 
\

\noindent
{\em Case 2.} $\nu(k^{\times})\subsetneq \nu(E^{\times})$. 
Then $\nu(E^{\times})/\nu(k^{\times})$ 
has $\mathbb Q$-rank $1$, i.e., for $y,z\in E^\times$ with nonzero 
$\nu(y),\nu(z)\in\nu(E^{\times})/\nu(k^{\times})$ 
there are nonzero $n_y,n_z\in \Z$ 
such that $n_y\nu(y)=n_z\nu(z)$. 
Indeed, $y,z$ define a finite algebraic extension $k_{y,z}(x)|k(x)$, 
hence $\nu$ is nontrivial on $k(x)$,
and the group 
$$
\nu(k_{y,z}(x)^{\times})/ \nu(k(x)^{\times})
$$ 
is finite.
Let $g\in k(x)^\times$ be such that the image of $\nu(g)$ in 
$\nu(E^{\times})/\nu(k^{\times})$ is infinite. Then for any 
$\sum_{i=0}^n a_ig^i$, with $a_i\in k$,  
$$
\nu(\sum_{i=0}^n a_ig^i)=\min_i(\nu (a_ig^i)),
$$ 
since none of the monomials $a_ig^i$ have the same value under $\nu$. Thus, 
$$
\nu(k(g)^\times)=\nu(k^{\times})\times \langle \nu(g)\rangle,
$$ 
The extensions $k_{y,z}(x)|k(x)$ and  $k(x)|k(g)$ are
finite, thus
$$
\nu(k_{y,z}(x)^{\times})/(\nu(k^{\times})\times \langle \nu(g)\rangle)
$$ 
is also finite, which implies the result for $\nu(E^{\times})$. 
Since $\psi(k^\times)= 1$,
$\psi (k_{y,z}(x)^\times)$ is 
the product of a finite group and $\Z$.
In particular, $\psi (k_{y,z}(x)^\times)$ 
consists of algebraically
dependent elements.
Since $E$ is a union of subfields
$k_{y,z}(x)$, the same holds for $E^\times$.

Thus $\psi(E^\times/k^\times)$ 
coincides with the image of $\nu(E^{\times})/\nu(k^{\times})$. 
Since all elements in $\nu(E^{\times})/\nu(k^{\times})$ have the same powers we see 
that lifts of elements in $\psi(E^{\times})$ to $L^\times$ are algebraically dependent over $l$. 
\end{proof}

\section{Restriction to planes}
\label{sect:lines}

Here we study restrictions of homomorphisms
$$
\psi : \P(K)=K^\times/k^\times\ra A:= L^\times/\tilde{l}^\times,
$$
satisfying assumptions of Theorem~\ref{principal_theorem},
to projective planes $\Pi\subset \P(K)$. 

\begin{prop}
\label{prop:17}
Let $\Pi:=\Pi(1,x,y)\subset \P(K)$ be such that
$\psi(x)\not\sim \psi(y)$. Then one of the following holds:
\begin{itemize}
\item[(a)] $\psi_{\Pi}$ is injective. 
\item[(b)] There exists a line $\ml\subset \Pi$ such that $\psi_{\Pi}$ is constant on $\Pi\setminus \ml$.
\item[(c)] There exists a point $\mq\in \Pi$ such that $\psi_{\Pi}$ is constant on
$\ml\setminus \mq$, for every $\ml\subset \Pi$ passing through $\mq$.
\item[(d)] $k=\mathbb Q$, $\psi_\Pi$ is induced 
from 
$$
\bar\psi_{\Pi} : \P^2(\Z/p) \ra A, 
$$ 
via a lattice $B\subset \Q^3$, and $\bar\psi_{\Pi}$ is of type $(a), (b)$, or $(c)$. 
\end{itemize}
\end{prop}

\begin{proof}
Assume that $\psi_{\Pi}$ is not injective: there are distinct $x_1, x_2\in \Pi$,
with $\psi(x_1)=\psi(x_2)\neq 1$. Consider 
$$
\Pi_1:= x_1^{-1}\cdot \Pi=\Pi(1, 1/x_1, y/x_1),
$$
since $\psi(y)\not\sim \psi(1/x_1)$, 
$\Pi_1$ satisfies the conditions of the theorem; if it
holds for $\Pi_1$ then it holds for the initial $\Pi$.
Thus we may assume that 
$$
S_1:=\{x'\in \Pi \,\mid \, \psi(x')=1\}
$$
contains at least two elements. 
Consider the map 
$$
\psi_{\sim}: \P(K)\ra A_{\sim},
$$
with values in dependency classes: 
\begin{itemize}
\item
$\psi_{\sim}(x')=1$ if $\psi(x')=1$ ,
\item 
$\psi_{\sim}(x')=\psi_{\sim}(x'')$ 
iff $\psi(x'),\psi(x'')\neq 1$ and $\psi(x')\sim\psi(x'')$.
\end{itemize}
We record properties of $\psi_{\sim}$: 
\begin{itemize}
\item[(TI)]
For every $\ml\subset \Pi$ with
$\ml\cap S_1=\emptyset$, we have
$$
\{ \psi_{\sim}(x') \,\mid x'\in \ml\} = \{\psi_{\sim}(x'')\, \mid x''\in \Pi\setminus S_1\}, 
$$
in particular, $\psi(\ml)$ has 
has algebraically independent elements.
\item[(TC)]
For every $\ml\subset \Pi$  with $\ml\cap S_1\neq \emptyset$, 
$\psi_{\sim}$ is constant on
$\ml\setminus (\ml\cap S_1)$. 
\end{itemize}
Property (AI) from Section~\ref{sect:fields} relates 
$\psi_{\sim}$ and $\psi$.

\begin{lemm}
\label{lemm:one}
If $\ml\cap S_1=\emptyset$ and $x', x''\in \ml$ are such that $\psi(x')\sim\psi(x'')$ then 
$\psi(x')=\psi(x'')$. 
\end{lemm}

\begin{proof}
There is a $z\in\ml$ with $\psi(z)\not\sim
\psi(x'), \psi(x'')$. Since $z^{-1}\cdot \ml\cap S_1\neq \emptyset$, 
all values  of $\psi$ on $\ml(x'/z,x''/z)\setminus 1$ 
are algebraically dependent. By (AI), if  $\psi(x')\neq \psi(x'')$ then 
$\psi(x')/\psi(z)\not\sim\psi(x'')/\psi(z)$, a contradiction. 
\end{proof}

\begin{lemm}
\label{lemm:DE}
Let $\ml,\ml'\subset \Pi$ be disjoint from $S_1$, 
put $z:=\ml\cap \ml'$, 
and assume that there exist
$x\in \ml$ and $x'\in \ml'$ such that
$$
\psi(x)\sim \psi(x') , \psi(x)\neq \psi(x'),
\quad \text{ and }\quad \psi(x),\psi(x')\not\sim \psi(z).
$$ 
Let $y\in \ml$ and $y'\in \ml'$ 
be such that $\psi(y)\sim \psi(y')$. 
Then either 
\begin{itemize}
\item  $\psi(y)\neq \psi(y')$, or
\item 
$\psi(y)=\psi(y')=\psi(z)$.
\end{itemize} 
\end{lemm}

\begin{proof}
By the same argument as in  
Lemma~\ref{lemm:one}, using (AI),  
$$
\psi(x)/\psi(z)\sim \psi(y)/\psi(z), \quad \text{ and  } 
\psi(x')/\psi(z)\sim \psi(y')/\psi(z),
$$ 
If $\psi(y)=\psi(y')\neq \psi(z)$
then $\psi(x)/\psi(z)\sim \psi(x')/\psi(z)$, contradiction.
\end{proof}

Let $\{T_j\}_{j\in J}$ be the set of 
intersections of algebraic dependency classes in $\P(K)$ with $\Pi$.
Split $J=J_2\sqcup J_3$ and consider the decomposition 
\begin{equation}
\label{eqn:pi}
\Pi=S_1\sqcup S_2\sqcup S_3, \quad \text{with} \quad 
S_1= T_1, S_2=\sqcup_{j\in J_2} T_j, \quad S_3= \sqcup_{j\in J_3} T_j. 
\end{equation} 
For any such decomposition, the induced map
$$
\Psi=\Psi_{\Pi}: \Pi\ra \{1, 2,3\}
$$
factors through $\psi_{\sim}$ and satisfies the conditions of 
Proposition~\ref{prop:plane}. 
Thus $\Psi$ is induced from a trivial coloring, 
with $S_1$ not depending on 
the decomposition.  
Since there exist 
lines disjoint from $S_1$, and $S_1$ contains at least two points, 
it follows that either
\begin{enumerate}
\item[(B)] 
$S_1=\Pi\setminus \ml$, for some $\ml\subset \Pi$,
and we are in Case (b), or
\item[(C)] 
$S_1=\cup_{i\in I} (\ml_i\setminus \mq)$, for some  
$\mq\in \Pi$ and $\ml_i$ through $\mq$, 
and we are in Case (c), or 
\item[(D)] 
$k=\mathbb Q$, and $\Psi$ is induced from a trivial coloring on 
$\P^2(\Z/p)$.
\end{enumerate}
Note that in Case (B), $\psi\equiv 1 $ on the affine plane
$\Pi\setminus \ml$.

\begin{lemm}
\label{lemm:conc}
In case $\mathrm{(C)}$, $\psi$ is constant on an affine plane, or
on $\ml_i\setminus \mq$, for all $i\in I$.
\end{lemm}

\begin{proof}
Consider $x\in \Pi\setminus (S_1\cup \mq)$ and lines $\ml$ 
containing $x$ but not $\mq$. Then $\psi_{\sim}\equiv \psi_{\sim}(x)$ on  
on $\ml\setminus (\ml\cap S_1)$.
Since $S_1$ is not an affine plane,
there is an $x'\in \Pi\setminus (S_1\cup \ml(x,\mq))$.
We have $\psi_{\sim}(x)=\psi_{\sim}(x')$. The union of
lines  $\ml \subset \Pi,\mq\notin \ml$,
through $x$, $x'$, is  equal to
$\Pi\setminus \mq$.
Thus $\psi_{\sim}$ takes only three values $\{ 1,\psi(x),\psi(\mq)\}$
and is constant on $\Pi\setminus (S_1\cup \mq)$.
Lemma~\ref{lemm:DE}, applied to $\ml$ through $\mq$, implies that
$\psi$ is constant on $\ml\setminus \mq$.
\end{proof}

We are left with Case (D), when $\Psi$
is induced via some
$$
\rho: \Pi=\P^2(\Q)\to \P^2(\Z/p)
$$
from a trivial coloring 
$$
\mathfrak  c: \P^2(\Z/p)\to \{1,2,3\},
$$
in the sense of Proposition~\ref{prop:plane}.
Put
$$
\bar S_i =\mathfrak c^{-1}(i), \quad i=1,2,3.
$$ 
Note that $S_1$ is a  finite union of subsets
$\Z_{(p)}+\Z_{(p)}$ and does not contain a complete line $\ml$. 
Consider shifts $\Pi_z:=z^{-1}\cdot \Pi$, for $z\in \Pi$.

\begin{lemm}
\label{lemm:piz-rest}
For every $z\in \Pi$, the restriction of $\psi_{\sim}$ to $\Pi_z$ is induced from 
$\P^2(\Z/p)$. 
\end{lemm}

\begin{proof}
We subdivide (D) into subcases:
\begin{enumerate}
\item[(D1)] 
For every $z$ and every splitting $\Pi_z=S_1\sqcup S_2\sqcup S_3$, 
where $S_2, S_3$ are unions of algebraic dependency classes, the 
set $\bar S_1\subset \P^2(\Z/p)$ is either a point, 
an affine line, or an affine plane.
\item[(D2)] 
Otherwise: for some $\Pi_z$ this is not the case. 
\end{enumerate}
First we treat (D1). 
Fix $\Pi=\Pi_z$ and a decomposition 
$\Pi=S_1\sqcup S_2\sqcup S_3$; we have 
$$
\bar\Psi: \P^2(\Z/p)\to \{ 1,2,3\},
$$
and
$$
\P^2(\Z/p)=\sqcup_{i=1}^3 \bar{S}_i,  \quad S_i=\rho^{-1}(\bar{S}_i) \subset \Pi. 
$$ 
By assumption (D1), we have 3 cases.
\begin{itemize}
\item $\bar{S}_1=\bar\mq$, for some $\bar\mq\in \P^2(\Z/p)$. 
For $x\in \Pi\setminus S_1$ and 
$\ml= \ml(\mq,x)$, with $\rho(\mq)=\bar\mq$, 
$\psi_{\sim}$ is constant on $\ml\setminus (\ml\cap S_1)$, by (TC). 
Apply this to all $\ml(\mq,x_1)$, where
$x_1$ runs over $S_1$, to conclude 
that $\psi_{\sim}$ is constant on preimages of affine lines
$(\bar \ml\setminus \bar\mq)$, with $\bar\mq\in \bar\ml$,
hence is induced from $\P^2(\Z/p)$.
\item 
$\bar{S}_1=\bar\ml\setminus \bar\mq$, for some 
$\bar\ml\subset \P^2(\Z/p)$ and $\bar\mq\in \bar\ml$.
Then $\bar{S}_1, \bar{S}_2$ and $\bar{S}_3$ form a flag on $\P^2(\Z/p)$: all points projecting to $\P^2(\Z/p)\setminus 
\bar{\ml}$ belong to the same algebraic dependency class because each pair of such points can be connected by 
a pair of lines which intersect $S_1$. 
Lemma~\ref{lemm:eight}
reduces the proof to the previous case, 
after changing to a different $\psi_{\sim}$-compatible lattice. 
\item 
$\bar{S}_1=\P^2(\Z/p)\setminus \bar\ml$, for some line $\bar\ml\in \P^2(\Z/p)$.
This reduces to the case $\bar{S}_1=\bar\mq$. 
\end{itemize}

We pass to (D2) and fix a plane $\Pi=\Pi_z$, with  
a splitting 
$$
\Pi=S_1\sqcup S_2\sqcup S_3,
$$
violating (D1). Then there exist a point $\bar\mq\in \P^2(\Z/p)$ 
and a set $\{\bar\ml_i\}_{i\in I}$ of at least two lines passing through $\bar\mq$ such that
$\bar{S}_1=\cup_{i\in I}  (\bar\ml_i\setminus \bar\mq)$. 
Moreover, we may assume that $\bar{S}_2=\bar\mq$, 
then $\bar{S}_3$ has the same structure as $\bar{S}_1$, 
i.e., a union of affine lines containing $\bar\mq$ in their closure.  

We claim that $\psi_{\sim}$ is constant on $S_3$: consider $\bar\mq_3,\bar\mq_3'\in \bar{S}_3$ not lying 
on a line through $\bar\mq$. Let $\mq_3, \mq_3'$ be any points projecting to $\bar\mq_3,\bar\mq_3'$. 
Since $\bar\ml(\bar\mq_3,\bar\mq_3')\cap \bar{S}_1\neq \emptyset$,  
the line $\ml(\mq_3,\bar\mq_3')$ intersects $S_1$, thus $\mq_3\sim \mq_3'$. 
By assumption on $\bar{S}_3$, any 
two points  in $S_3$ can be connected by a chain of such lines.

Note  that $\psi_{\sim}$ is constant on $S_2$: consider 
$$
\mq_1,\mq_2\quad \text{ with }\quad \rho(\mq_1)=\rho(\mq_2)=\bar\mq\in \bar{S}_2.
$$ 
Then $\psi(\mq_1)=\psi(\mq_2)$.
Indeed, consider $\ml_5= \ml(\mq_1, x_1)$ and $\ml_6=\ml(\mq_2,x_2)$, 
where $\rho(x_i)=\bar x_i \in \bar S_1, \bar x_1\neq \bar x_2$.
Hence $\mq_3:=\ml_5\cap \ml_6$ projects to $\bar\mq$.
Thus $\psi(\mq_1)=\psi(\mq_3)= \psi_{\sim}(\mq_2)$.
Thus $\psi_{\sim}$ is constant on $S_2$,
hence $\psi_{\sim}$ is induced from $\P^2(\Z/p)$.
\end{proof}

\begin{lemm}
\label{lemm:anew}
The map $\psi_{\Pi}$ is induced from $\bar{\psi}_{\Pi}:\P^2(\Z/p)\ra A$
which is of the type (a), (b), or (c). 
\end{lemm}

\begin{proof}
By Lemma~\ref{lemm:piz-rest},
we have the following possibilities:
\begin{enumerate}
\item
$\psi_{\sim}$ is induced from  a flag map on $\P^2(\Z/p)$
and we can assume that $\bar S_1=\bar\mq$, by Lemma~\ref{lemm:eight};
\item
$\psi_{\sim}$ is induced from a map on
$\P^2(\Z/p)$ which is constant on affine lines $\bar\ml_i\setminus \bar\mq$, with $\bar\mq\in \bar\ml$, 
and $\bar S_1=\bar\mq$;
\item
$\psi_{\sim}$ is induced from a map on
$\P^2(\Z/p)$ which is constant on affine lines
$\bar \ml_i\setminus \bar\mq$, with $\bar\mq\in \bar\ml$, 
and $\bar S_1$ contains $\bar \ml_i\setminus \bar\mq,i=1,2$.
\end{enumerate}

\

{\it Case (1)}: 
We may assume that $\bar S_3=\P^2(\Z/p)\setminus \bar\ml$, for some $\bar\ml$ 
with $\bar\mq\in\bar\ml$, and $\bar\ml\setminus\bar\mq =\bar{S}_2$. 
Let  $\ml$ be disjoint from $S_1$ and $\mq,\mq'\in \ml\cap S_3$.  
Since $\psi_{\sim}(\mq)=\psi_{\sim}(\mq')$
and $\ml$ intersects $S_2$, $\psi(\mq)=\psi(\mq')$, 
by Lemma~\ref{lemm:one}. 
Since any two points in $S_3$ 
can be connected by a chain
of lines disjoint from $S_1$, $\psi$ is constant on $S_3$. It is also 
constant on $\rho^{-1}(\bar\mq_2)$, for $\bar\mq_2\in \bar S_2$.
Indeed, if $\mq_2,\mq_2'$ are distinct points projecting to $\bar\mq_2$
and $\ml, \ml'$ lines containing $\mq_2$, resp. $\mq_2'$, 
avoiding $S_1$ and projecting to distinct lines in $\P^2(\Z/p)$,
then $\mq_2'':=\ml\cap \ml'$ also projects to $\bar\mq_2$.
Thus $\psi(\mq_2)=\psi(\mq_2')=\psi(\mq_2'')$.

\

{\it Case (2)}:  $\bar S_1= \bar\mq$.
If $\psi$ is induced from a noninjective 
$\bar\psi:\P^2(\Z/p)\ra A$, $\psi$ is constant on the preimage
of every affine line $\bar \ml \setminus \bar \mq$, by 
the same analysis over a finite field.

If there exist $y_1,y_2$, projecting
to the same points $\bar x \in \bar \ml\setminus \bar\mq$,
with $\psi(y_1)\neq \psi(y_2)$, let $z_1,z_2$ have
$\psi_{\sim}(z_i)\neq \psi_{\sim}(y_i)$. Consider 
$$
z:=\ml(y_1,z_1)\cap \ml(y_2,z_2),
$$ 
so that $\rho(z)=\bar{x}$.  
Then $\psi(y_1) = \psi(z)=\psi(y_1)$, by Lemma~\ref{lemm:DE}.
Since all points over $\bar x$ are connected by a chain
of lines of such type,
$\psi$ is constant on $\rho^{-1}(\bar x)$.

\

{\it Case (3)}: The argument of Case (1) shows that
$\psi$ is constant on the preimage of any affine line  
$\bar \ml \setminus \bar \mq$ contained in $\bar S_3$. 
Indeed, let $z_1,z_2 \in S_3$ and consider $\ml:=\ml(z_1,z_2)$.
It intersects $S_2$ and hence $\psi(z_1)=\psi(z_2)$.
Thus $\psi$ is induced from $\P^2(\Z/p)\setminus \mq =\bar S_2$.
Let $\mq,\mq'$, projecting $\bar\mq$. Consider lines
$\ml(\mq,z_1)$ and $\ml(q',z_2)$ with $z_i\in S_3$, which intersect
in $\mq'',\rho(\mq'')=\bar\mq$.
Then  $\psi(\mq) = \psi(\mq'')=\psi(\mq')$, by Lemma~\ref{lemm:DE}.
Since any pair of points over $\bar q$ can be connected
by a chain of such lines, $\psi$ is constant on $\rho^{-1}(\bar\mq)$.
\end{proof}

This concludes the proof of Proposition~\ref{prop:17}. \end{proof}


\begin{rema} 
This Lemma is similar to \cite{HalesStraus} 
and \cite[Lemma 13]{GTPG}. 
\end{rema}

\section{Lines of injectivity}
\label{sect:new-lines}

In our analysis of the restriction $\psi_{\ml}$ of 
$$
\psi: \P(K)\ra A=L^\times/\tilde{l}^\times
$$
to lines $\ml=\ml(1,x)\subset \P(K)$ we distinguish the following
possibilities:
\begin{itemize}
\item $\psi_{\ml}$ is not induced from a map $\bar{\psi}_{\ml}:\P^1(\Z/p)\ra A$
and $\psi_{\ml}$ is:
\begin{itemize}
\item[(I)] injective
\item[(N)] not injective and nonflag
\item[(F)] a nonconstant flag map
\end{itemize}
\item 
$\psi_{\ml}$ is induced from $\bar{\psi}_{\ml}:\P^1(\Z/p)\ra A$ and $\bar{\psi}_{\ml}$ is
\begin{itemize}
\item[($\bar{\mathrm I}$)] injective
\item[($\bar{\mathrm N}$)] not injective and nonflag
\item[($\bar{\mathrm F}$)] a nonconstant flag map
\end{itemize}
\item (C) $\psi_{\ml}$ is constant
\end{itemize}

\begin{defn}
\label{defn:u}
Let
$
\mathfrak u\subset \P(K)
$ 
be the union of all lines through $1$, on which $\psi$ is injective and put
$$
\mathfrak U:=\{xy~|~x,y\in\mathfrak u\}\subseteq \P(K).
$$
\end{defn}

\begin{lemm} 
\label{three-on-u} 
If $\psi(\mathfrak u)$ contains at least two 
algebraically independent elements, then $\mathfrak U$ is a group. 
\end{lemm}

\begin{proof}
Clearly, $\mathfrak u$ and $\mathfrak U$
contain $1\in K^{\times}/k^{\times}$.
If $x\in \mathfrak U$ then $x^{-1}\in \mathfrak U$, by 
injectivity of $\psi$ on $\ml(1,x^{-1})$. 
Furthermore, 
\begin{equation} 
\label{u-ratio} \text{$xy^{-1}\in\mathfrak u$, for all 
$x,y\in\mathfrak u$ such that $\psi(x)\neq\psi(y)$.}
\end{equation} 
Indeed, if $\psi(x)\not\sim\psi(y)$ then
$\psi$ is injective on $\Pi(1,x,y)$, 
by Proposition~\ref{prop:17}, and in particular on  
$\ml(x,y)=y\cdot\ml(1,xy^{-1})$;
thus, $xy^{-1}\in\mathfrak u$. 

If $\psi(x)\sim\psi(y)$, but are not equal in $A$,
take $z\in\mathfrak u$ such that $\psi(x)\not\sim\psi(z)$. 
Then $x/z,y/z\in\mathfrak u$, as above. Since 
$\psi(x/z)\not\sim\psi(y/z)$, the same argument shows 
that $(x/z)/(y/z)=xy^{-1}\in\mathfrak u$, proving \eqref{u-ratio}.

To show that $\mathfrak U$ is multiplicatively closed it
suffices to check that for every 
$x_1,x_2,x_3\in\mathfrak u\setminus\{1\}$ there exist  
$s_1,s_2\in\mathfrak u$ with $x_1x_2x_3=s_1s_2$. 
Note that $\psi(x_ix_j)\neq 1$ for some $1\le i<j\le 3$. (Otherwise, 
$$
\psi(x_1x_2)=\psi(x_1x_3)=\psi(x_2x_3)=1,
$$
and therefore, 
$$
\psi(x_1^2)=\psi((x_1x_2)(x_1x_3)/(x_2x_3))=1,
$$
so $\psi(x_1)=1$.) 
Then, by (\ref{u-ratio}), $x_ix_j\in\mathfrak u$, so we can 
take $s_1:=x_ix_j$ and $s_2:=x_t$, where $\{i,j,t\}=\{1,2,3\}$. 
\end{proof}

\begin{defn}
\label{defn:baru}
Let $\bar{\mathfrak u}\subset \P(K)$ 
be the union of all lines $\ml$ through 1, such
that the restriction of $\psi$ to $\ml$ is induced via 
an injective map
$$
\bar{\psi}_{\ml}:\P^1(\Z/p)\ra A, 
$$ 
and put
$$
\bar{\mathfrak U}:=\{ xy\, \mid\, x,y\in \bar{\mathfrak u}\}\subseteq \P(K).
$$
\end{defn}

\begin{lemm}
\label{lemm:up}  
If $\psi(\bar{\mathfrak u})$ contains at least
two algebraically independent elements, 
then $\bar{\mathfrak U}$ is group.
\end{lemm}

\begin{proof}
The proof follows the same steps as the proof of 
Lemma~\ref{three-on-u}.
\end{proof}

\begin{lemm}
\label{lemm:restriction}
Assume  $\P(K)$
contains lines of type $(\mathrm I)$ and one of the types 
\begin{equation}
\label{eqn:types}
(\mathrm N), (\bar{\mathrm I}), (\bar{\mathrm N}), \quad or  \quad (\bar{\mathrm F}).
\end{equation}
Then there exists a one-dimensional 
subfield $E\subset L$ such that for all lines $\ml\subset \P(K)$ of 
type $(\mathrm I)$, $(\mathrm N)$, $(\bar{\mathrm I})$, 
$(\bar{\mathrm N})$, or $(\bar{\mathrm F})$ we have
$$
\psi(\ml)\subset E^\times/\tilde{l}^\times.
$$
In particular, if $\psi(\mathfrak u)$ 
contains algebraically independent elements, 
lines of type  $(\mathrm N)$, $(\bar{\mathrm I})$, $(\bar{\mathrm N})$, and $(\bar{\mathrm F})$
do not exist. 
\end{lemm}

\begin{proof}
Let $\ml=\ml(1,y)$ be a line of type (I). 

If there exists another line $\ml(1,y')$
of type (I) with $\psi(y)\not\sim \psi(y')$, i.e., $\psi(\mathfrak u)$ contains
independent elements,  
then lines of the listed type cannot exist, 
indeed, if $\ml(1,x)$ is of types listed in \eqref{eqn:types}, 
we apply Proposition~\ref{prop:17} to $\Pi=\Pi(1,x,y)$.  
In Case (a), the exceptional line is $\ml(1,y)$ and hence
the restriction of $\psi$ to any other line is either constant or of type (F),
contradiction. In Case (b), all lines are either of type (I) or (F), again
a contradiction. Case (c) does not apply, since $\ml(1,y)$ is not induced
from a map $\P^2(\Z/p)\ra A$. Contradiction.  

If $\psi(\mathfrak u)$ does not contain algebraically independent elements, but one of 
the lines $\ml(1,x)$ in \eqref{eqn:types} is such that $\psi(x)\not\sim\psi(y)$
then we apply the same argument to $\Pi(1,x,y)$ and obtain the same contradiction. 
\end{proof}

\begin{lemm}
\label{lemm:24}
Assume that $\psi(\mathfrak u)$ contains algebraically independent elements. 
Consider $\ml:=\ml(1,y)\not\subseteq \mathfrak u$ and assume that
$\ml\cap \mathfrak U$ consists 
of at least two points $1,z'$. 
Then $\ml\cap \mathfrak U$ is either $\ml$ or  
$\ml\setminus \mq$, for some point $\mq\in \ml$. 
\end{lemm}

\begin{proof}
Assume that $\psi_{\ml}$ is not constant, e.g., $\psi(y)\neq 1$.  
By assumption, there is an $x$ with $\ml(1,x)\subset\mathfrak u$ with
$\psi(x)\not\sim \psi(y)$. We apply Proposition~\ref{prop:17} to 
$\Pi:=\Pi(1,x,y)$. We are not in Case (c) of this lemma. 
If we are in Case (a), then 
$\psi$ is constant on $\Pi\setminus \ml(1,x)$, which implies that $\ml$ is
of type (F). If we are in Case (b), then the exceptional point $\mq=y$, 
and $\psi$ is constant, on the complement to $\mq$,  
on every line through $\mq$, thus $\ml$ is of type (F).  

Put $z'= t/t'$, 
with $t,t'\in \mathfrak u$. 
If $\psi(t)\neq \psi(t')$ 
then Equation~\eqref{u-ratio} implies that $z'\in 
\mathfrak u$, a contradiction. 
Thus $\psi_{\ml}$ is either constant or contains one point 
$y'\notin \mathfrak U$. In Case (a), $\psi$ is constant on 
$\Pi\setminus \ml(1,x)$, thus identically 1 on the line $\ml$. 
In Case (b), $\psi$ is injective on every line 
not containing the exceptional point $\mq$, in particular on 
$\ml(1,t'/t'')$, for all $t''$, thus 
$t'/t''\in \mathfrak u$, thus $t''\in \mathfrak U$. 
Taking $t''\in \ml\setminus \mq$ we obtain the claim.

Now assume that $\psi_{\ml}$ is constant. We claim that   
$\ml\setminus (\ml\cap  \mathfrak U)$ 
contains at most one point. 
Assume otherwise, and pick $w_1,w_2$ in this set. 
Note that $\psi$ is injective on 
every line 
$\ml(u',t')\subset \Pi$, with $t'\in \Pi(1,x,y)\cap \mathfrak u$, $t'\neq 1$, 
and any point $u'\in \ml\cap \mathfrak U$.     
Indeed, we can represent $u'=w/w'$, with $w,w'\in \mathfrak u$ and
with $\psi(w)= \psi(w')\not\sim\psi(t')$. 
Then $t'w'/w\in \mathfrak u$ and 
$\ml(t'w'/w, 1)\subset \mathfrak u$.
The converse is also true, and $(\Pi\setminus \ml)\subset \mathfrak u$.
Indeed, consider lines  through $u'$ which are not equal to $\ml$; 
$\psi$ is injective on such lines. 

Now consider two families of lines: those 
passing through $w$ (except $\ml$), 
and those throgh $w'$ (again, except $\ml$). 
All such lines are of type (F), with generic value $\neq 1$, since $\psi$ 
does not take value 1 on $\Pi\setminus \ml$. 
Consider lines $\ml(w,v)$ and $\ml(w',v)$ from these families, with 
$v\in (\Pi\setminus \ml)$. The generic $\psi$-value on these lines 
is the same and equal $\psi(v)$. 
A line through $u'$, which does not contain $v$ 
cannot be of type (I), since it intersects lines 
$\ml(w,v)$ and $\ml(w',v)$  in distinct points, but taking the same value
on these points, contradicting the established fact that such lines are of type (I).   
\end{proof}

\begin{lemm}
\label{lemm:restriction2}
Assume  $\P(K)$ contains lines 
of type  $(\bar{\mathrm I})$ and 
there exist lines of type $\mathrm{(I)}$, or
$\mathrm{(N)}$, or $(\bar{\mathrm N})$.  
Then $\psi(\bar{\mathfrak u})$ 
does not contain algebraically independent elements. 
\end{lemm}

\begin{proof}
Assume the contrary. Let $\ml(1,x)$ be a line of type (I) or (N).  
Then there exists an $y\in \bar{\mathfrak u}$ 
such that $\psi(y)\not\sim\psi(x)$. 
We apply Proposition~\ref{prop:17} to $\Pi=\Pi(1,x,y)$ and
obtain a contradiction as in the proof of Lemma~\ref{lemm:restriction}.

Let $\ml(1,x)$ be of type ($\bar{\mathrm N}$).
We claim that $\Pi$ does not contain lines of type (F).  
To exclude this possibility, let $\ml=\ml(z,t)\in \Pi$ be such 
a line with generic $\psi$-value equal to $s\in A$.

Take points $x_1,x_2\in \ml(1,x)$ such that
$\psi(x_1)\neq \psi(x_2)$, this is possible since $\psi$ 
takes at least two values on $\ml(1,x)$. Choose 
$y_1,y_2\in \ml(1,y)$ and  $\psi(y_1)\neq \psi(y_2)$ and are both not equal to 
$1\in A$, this is possible because $\psi$ takes at least three values on 
$\ml(1,y)$ which is of type ($\bar{\mathrm I}$).

\centerline{
\xymatrix{
y_2 \ar@{-}[d]        & & \mathfrak{q}  \\
y_1 \ar@{-}[d]    & &                     \\
1 \ar@{-}[r]              &  x_1  \ar@{-}[r] & x_2
}}

\

Moreover, we can  assume that the lines $\ml_{ij}:=\ml(x_i,y_j)$ do not pass through
the distinguished point $\mq\in \ml(z,t)$ (where $\psi$ takes the nongeneric
value). Thus $\ml_{ij}:=\ml_{ij}\cap \ml(z,t)$ is a generic point of $\ml(z,t)$, 
which differs from $x_1,x_2,y_1,y_2$. Then 
$$
\frac{\psi(x_i)}{s} \sim \frac{\psi(y_1)}{s} \sim \frac{\psi(y_2)}{s},  
$$
for both $i=1,2$. 
Hence 
$$
1\neq \frac{\psi(x_1)}{\psi(x_2)}\sim  \frac{\psi(y_1)}{\psi(y_2)} \neq 1
$$
Therefore, $\psi(x)\sim\psi(y)$, contradiction.

Thus, for every $\ml\subset \Pi(1,x,y)$ the restriction 
$\psi_{\ml}$ is induced 
from a map $\bar{\psi}_{\ml}:\P^1(\Z/p)\ra A$.  
Now we apply Lemma~\ref{lemm:anew}. 
In Case (a) of that Lemma, 
the exceptional line is $\ml(1,y)$ and of type 
($\bar{\mathrm I}$) and hence
the restriction of $\psi$ to any other line is either constant or of type 
($\bar{\mathrm F}$),
contradiction the assumption that $\ml(1, x)$ is of type ($\bar{\mathrm N}$). 
Cases (b) and (c) are excluded: $\psi$ is not induced from an 
injective map, nor a flag map on $\ml(1,x)$.  
\end{proof}

\begin{lemm}
\label{lemm:N}
Assume that the pair of lines $(\ml(1,x), \ml(1,y))$ is of one of the following types
$$
(\mathrm{N}, \mathrm{N}), \quad (\mathrm{N}, \bar{\mathrm{N}}),
(\mathrm{N}, \bar{\mathrm{F}}), \quad (\bar{\mathrm{N}}, \bar{\mathrm{N}}).
$$ 
Then $\psi(x)\sim \psi(y)$. 
\end{lemm}

\begin{proof}
Follows from the same arguments as in Lemma~\ref{lemm:restriction2}
and Lemma~\ref{lemm:restriction}.
\end{proof}

\begin{lemm}
\label{lemm:25}
Assume that $\psi(\bar{\mathfrak u})$ contains algebraically independent 
elements. Consider  
$\ml:=\ml(1,z)\not \subseteq \bar{\mathfrak u}$, and  assume that
$\ml\cap \bar{\mathfrak U}$ consists of 
at least two points $1,z'$.
Then $\psi(z')=1$ and $\ml\cap \bar{\mathfrak U}$ is either
\begin{enumerate}
\item $\ml$;
\item 
an affine line, with $\psi$ not constant on $\ml$;
\item projectively equivalent
to $\Z_{(p)}\subset \P^1(\Q)$;
\item 
an affine line and $\psi$ is constant on $\ml$.
\end{enumerate}
\end{lemm}

\begin{proof}
Assume that $\ml\notin$ (C). Write 
$z'= x/x'$ with $x,x'\in \bar{\mathfrak u}$. 
\begin{itemize}
\item 
If $\psi(x)\neq \psi(x')$ 
then $x/x'\in \bar{\mathfrak u}$, by Equation~\ref{u-ratio}, 
thus $\ml\subset \bar{\mathfrak u}\subset \bar{\mathfrak U}$,
contradiction, so that $\psi(z')=1$.
\item 
If $\psi(x)= \psi(x')\neq 1$,
let $t\in \bar{\mathfrak u}$ be such that $\psi(t)\not\sim \psi(x')$,
and is also independent from a nontrivial value on $\ml$.
Then $t/x' \in \bar{\mathfrak u}$ and the restriction of 
$\psi$
to (a shift of) $\ml(t,x/x')$ is of type ($\bar{\mathrm I}$).
In particular, $\ml(1,t),\ml(t,x/x')$ are also 
of type ($\bar{\mathrm I}$), by the same argument as in the 
proof of Lemma~\ref{lemm:restriction2}.
\end{itemize}
This lemma implies that $\ml$ is of type (F) or ($\bar{\mathrm F}$).

\begin{itemize}
\item 
$\ml\in$ (F).  In the notation of  
Proposition~\ref{prop:17}, 
$\psi$ is of type b) on $\Pi(1,z,t)$ and
the restriction of $\psi$ to every line in $\Pi(1,z,t)$, not passing
through a distinguished point $\mq\in \ml$, with $\psi(\mq)\neq 1$, is
of type ($\bar{\mathrm I}$), 
which implies that $\ml\setminus \mq \subset \mathfrak U$, 
i.e., we are in Case (V).  
\item 
$\ml\in$ ($\bar{\mathrm F}$). 
In this case,  $\Pi(1,t,z)$  does not contain
lines of type (F), because otherwise, by
Proposition~\ref{prop:17}, $\ml$ will also be of type (F).
Hence $\psi$ is induced from $\P^1(\Z/p)$ on any line
in $\Pi(1,t,z)$ and there are two independent values
of $\psi$ on $\Pi(1,z,t)$ not equal to 1.
Then $\psi$ on $\Pi(1,z,t)$
is induced from $\bar{\psi}: \P^2(\Z/p)\ra A$,
by Lemma~\ref{lemm:anew}. 

The map $\bar\psi$ is injective on 
$\ml(1,t')$ and $\ml(t',z'_1)$, where
both $t', z'_1$ are the images of $t,z'$
under the reduction map, and a flag map
on  $\ml(1,z_1)$, where $z_1$ is the image of $z$
in $\P^2(\Z/p)$.
Thus $\psi$ is induced from type b),
and hence $\mathfrak U\cap \ml$
consists of $y$, with $\psi(y)=1$, a set projectively equivalent
to $\Z_{(p)}\subset \P^1(\Q)$,  
and we are in Case (P).
\end{itemize}

Assume that $\ml\in$ (C). Here the difficulty is that
$\psi(\Pi(1,z,t))$
does not contain algebraically independent elements
and we cannot apply Lemma~\ref{lemm:anew}.
Note that $\ml(t,s)$, for $s=r/r',r,r'\in \mathfrak u$, $s\in \ml$,
are of type ($\bar{\mathrm I}$), 
by the argument above. 

Then any line $\ml(t',u')\subset \Pi(1,z,t)$,
with $\psi(t')\neq \psi(u')$, is of type ($\bar{\mathrm I}$), 
since $\psi$ takes at least three values
on this line. Hence $s:=\ml(t',u')\cap \ml\in \bar{\mathfrak U}$.

On the other hand, if $s'\in \ml$ is 
not in $\bar{\mathfrak U}$ then there are at most two values
on any line containing $s'$, including $\psi(s')=1$.
We split all points into subsets:
\begin{enumerate}
\item $S_T:=\{ x\, |\, \psi(x)\neq 1\}$;
\item $S_1:=\{ x\in \bar{\mathfrak U} \,|\, \psi(x)=1\}$;
\item $S_2:=\{ x\notin \bar{\mathfrak U}\, |\,  \psi(x)=1\}$.
\end{enumerate}
Note that $S_T, S_1\neq \emptyset$. 
If $S_2=\emptyset$ then $\ml\subset \bar{\mathfrak  U}$; 
and $\bar{\mathfrak U}\cap \ml$ satisfies the lemma.  

Assume that $S_2\cap \ml \neq \emptyset$.
We claim that every line in $\Pi(1,z,t)$ lies in the union of two
of such subsets. 
Clearly, this holds for $\ml$. 
Let $\ml'\subset \Pi(1,z,t)$ be a different line and put $s:=\ml\cap \ml'$. 
If $s\in \bar{\mathfrak U}$, then $\ml(s,t)\subset \bar{\mathfrak U}$, by construction, 
and all points $s\in \ml\cap S_T$ 
are in $\bar{\mathfrak u}$ and those with $\psi(s)=1$
in $\bar{\mathfrak U}$.
In particular, $\ml(s,t)\subset S_T\sqcup S_1$.
If $s'\in \ml$ is in $S_2$, then 
$\ml(s',x)$ is of type (F), ($\bar{\mathrm F}$) or (C), 
and hence $\psi$ takes at most two values on  $\ml(s',x)$, including $\psi(s')=1$.

If $s_2\in \ml(s',x),\psi(s_2)= 1, x\in \bar{\mathfrak u},\psi(x)\neq 1$, then
$s_2\in S_2$. Otherwise, if $s_2\in \mathfrak U, x\in\bar{\mathfrak  u}$, and then
$\psi$ is injective on $\ml(s',x)=\ml(s_2,x)$,
by the argument above.
Hence $s_2\in S_2$. Thus $\ml(s_2,t')$, with $t'\in \ml(s,t)$,
is contained either in $S_2\sqcup S_T$ or $S_1\sqcup S_2$. 

Any $y\in \Pi(1,z,t)$, with $\psi(y)\neq 1$, 
is contained
in $\bar{\mathfrak u}$. Indeed, consider
$\ml(y, y')$, with $\psi(y)\neq\psi(y'), y'\in \ml(t,s),\psi(y')\neq 1$, 
and $s_y:=\ml(y, y')\cap \ml(1,z)$. 
Then $\psi(s_y)=1$, hence $y'/s_y \in\bar{\mathfrak u}$, 
and $\psi$ is injective on
$\ml(y,y')$. Since $y'\in \bar{\mathfrak u}$, we 
find that $y\in \bar{\mathfrak u}$ 
and $s_y\in S_1$.

Thus $S_T\subset \bar{\mathfrak u}$ and any line
$\ml(y,s),$ with $\psi(s)=1$, is either contained in
$S_2\sqcup S_T$ or in $S_T\sqcup S_1$.
This implies that
any $\ml(s,s_2)$, with $s\in S_1,s_2\in S_2$, is
contained in  $S_1\sqcup S_2$.
Note that none of the lines is contained
in one of the subsets $S_T,S_1,S_2$.
By Proposition~\ref{prop:main}, 
the decomposition $\Pi=S_T\sqcup S_1\sqcup S_2$
is either
\begin{enumerate}
\item 
a cone
over the decomposition of $\ml(t,s)$ into the
intersection with $S_T$ and $S_1$, and 
$S_2$ is just one point in $\ml$;
\item  or is induced
from a decomposition of $\P^2(\Z/p)$ over
the residue of $\ml$,
with $S_1$ equal to the preimage
of a point, and hence $S_2\cap \ml$
is projectively equivalent to $\Z_{(p)}$.
\end{enumerate}
\end{proof}


\section{Proof of the main theorem}
\label{sect:proof}

We turn to the proof of Theorem~\ref{principal_theorem}, describing 
homomorphisms
$$
\psi: \P(K)\ra \P(L)
$$
preserving algebraic dependence.  
There are two possibilities:
\begin{itemize}
\item[(V)] $\psi$ factors through a valuation,
\item[(P)] $\psi$ factors through a subfield,
\end{itemize}
described in detail in the Introduction. 

\

We organize our proof as a case by case analysis, based on 
types of line, introduced at the beginning of Section~\ref{sect:new-lines}.    
We consider two sets of cases as follows.
\begin{itemize}
\item 
Generic cases: 
$\psi(\mathfrak u)$ (respectively, $\psi(\bar{\mathfrak u})$), 
contains nonconstant algebraically independent elements, i.e., 
there exist $y_1,y_2\in \psi(\mathfrak u)$ (respectively, 
$\psi(\bar{\mathfrak u})$)
such that $y_1\not\sim y_2$. 
\item 
Degenerate cases: these sets do not contain algebraically 
independent elements.
\end{itemize}

In our proof we need the following technical assumption: 
\begin{itemize}
\item[(AD)]  $\psi(\bar{\mathfrak u})$ does not
contain nonconstant algebraically independent elements.
\end{itemize} 
This is satisfied when $K$ has positive characteristic.

\begin{lemm}
\label{lemm:subcases}
Assume that $\psi(\mathfrak u)$ contains nonconstant algebraically dependent elements and that 
$\P(K)$ contains  lines of type $(\mathrm F)$ and possibly also $(\mathrm C)$. 
Then there exists a valuation $\nu$ of $K$ such that 
$\mo_{\nu}^\times\subseteq \mathfrak U$ and $\psi((1+\mm_{\nu})^\times)=1$. 
\end{lemm}

\begin{proof}
By Lemma~\ref{three-on-u}, $\mathfrak U\subset K^\times/k^\times$ is a group, 
the induced quotient map 
$K^\times/k^\times\to K^\times/\mathfrak U$
is a nontrivial flag map, by the assumption on the existence of lines of type (F) in $\P(K)$. 
Thus there is a map 
$$
\mo_{\mu}^\times\to K^\times\to \Gamma_{\mu},
$$ 
for some valuation $\mu$, 
with the property
that  $K^\times \to K^\times/\mathfrak U$ is a composition 
$$
K^\times \to \Gamma_{\mu}\stackrel{r_{\mu}}{\lra} K^\times /\mathfrak U.
$$
Let 
$$
\Gamma_{\mu}^+:=  \nu(\mo_{\mu}\setminus 0)\subset \Gamma_{\mu}
$$ 
be the subsemigroup of positive elements and 
put 
$$
\Ker(r_{\mu})^+:=\Ker(r_{\mu})\cap \Gamma_{\mu}^+.
$$

\begin{itemize}
\item 
Assume that $\Ker(r_{\mu})^+=0$. Then for any nonconstant
$$
x\in \mo_{\mu}^\times/(k^\times\cap \mo_{\mu}^\times) \subset \mathfrak u, \quad 
y\in (\mm_{\mu}\setminus 0)/ (k^\times\cap \mo_{\mu}^\times),
$$
the restriction of $\psi$
to $\ml(x,y)$ is a flag map with generic value $1$
by Proposition~\ref{prop:17}, Case (c), with $y=\mq$, 
hence the result holds for $\nu=\mu$.
\item 
Assume that $\Ker(r_{\mu})^+\neq 0$. 
Assume in addition that there exists a $\gamma^+\in (\Gamma_{\mu}^+\setminus \Ker(r_{\mu})^+)$
and such that $\gamma^+<\gamma'$ for some $\gamma'\in \Ker(r_{\mu})^+$. 
Consider $x\in (\mathfrak u \setminus  1)$, 
with $\mu(x)= \gamma'$,  and $y\in \mo_{\mu}^\times/ (k^\times\cap \mo_{\mu}^\times)  $, with  
$\mu(y)=\gamma^+$. 
The restriction of $\psi$ to $\ml(1,y)\subset \P(1,x,y)$ is  a flag map with generic value $1$.
On the one hand, $\ml:=\ml(x,y)\not\subset\mathfrak u$, 
hence $\psi_{\ml}$ is a flag map, with generic value $\psi(x)$.
On the other hand, the generic value of $\psi$ on $\ml(1,y)$ is $1$,
hence $\psi(x+y)= \psi(x)$ and $x+y\in \mathfrak u$.
We have $\mu(y) < \mu(x)$ and, on $\ml(x,y)$, we have 
$\mu(x+y)=\mu(y)$, hence $\psi(x+y)=\psi(y)$, contradiction.

This implies that elements of $\Ker(r_{\mu})^+$ are smaller than all elements in $(\Gamma_{\mu}^+ \setminus \Ker(r_{\mu})^+)$. Thus 
the subgroup of $\Gamma_{\mu}$ generated by $\Ker(r_{\mu})^+$ is an ordered subgroup. 
The homomorphism $ \Gamma_{\mu}\to \Gamma_{\mu}/ \Ker(r_{\mu})^+$ identifies
$\Gamma_{\mu}/\Ker(r_{\mu})^+$ with a valuation group $\Gamma_{\nu}$ for some valuation $\nu$ of $K$, and 
$\psi((1+\mm_{\nu})^\times)=1$.
\end{itemize}
\end{proof}

We can also treat all degenerate cases, i.e., $\psi(\mathfrak u)$ and $\psi(\bar{\mathfrak u})$ do not contain 
nonconstant algebraically independent elements. 

\

{\em Most degenerate case: no $(\mathrm{I})$, $(\bar{\mathrm{I}})$, $(\mathrm{N})$, and 
$(\bar{\mathrm{N}})$-lines}: 
\begin{itemize}
\item
Then $\psi$ is a flag map on all $\ml\subset \P(K)$, hence a flag map,
and there exists a valuation $\nu$ such that $\psi$ factors through 
$\Gamma_{\nu}$, and we are in Case (V) of the Theorem~\ref{principal_theorem}, since $\psi(\mo_{\nu}^\times)=1$. 
\end{itemize}

\

{\em Degenerate case: no $(\mathrm{I})$ and $(\bar{\mathrm{I}})$-lines, but $(\mathrm{N})$ or $(\bar{\mathrm{N}})$-lines }: 
\begin{itemize}
\item 
If there exist (N) or $(\bar{\mathrm{N}})$-lines then, by Lemma~\ref{lemm:N}, 
there exists a 1-dimensional subfield $L_1\subset L$ such that the images of
all such lines are contained in $L_1^\times/l^\times$. Consider the induced projection homomorphism
$$
\psi_1:\P(K)\ra L^\times/l^\times \ra L^\times/L_1^\times.
$$
Note that the restriction of $\psi_1$ to any line $\ml\in \P(K)$ is a flag map, and there exist lines on which it is a nontrivial flag map, since the image of $\psi$ contains at least two algebraically independent elements. Thus there is a nontrivial valuation $\mu$ of $K$
such that $\psi_1$ factors through the value group $\Gamma_{\mu}$. 
\end{itemize}

\

{\em Degenerate case: 
there exist $(\mathrm{I})$-lines $\ml$ and $\psi(\ml)\subset L_1^\times/l^\times$, for 
some 1-dimensional field $L_1\subset L$.}

\begin{itemize}
\item
Let $L_2$ be the algebraic closure of $L_1$ in $L$.
There
may also 
exist lines $\ml\subset \P(K)$ of type (N),  $(\bar{\mathrm{N}})$,  $(\bar{\mathrm{I}})$, or  $(\bar{\mathrm{F}})$, 
with respect to $\psi$,  
but $\psi(\ml)\subset L_2^\times/l^\times$ for all such $\ml$, by Lemma~\ref{lemm:restriction}. 
Again, every $\ml\subset \P(K)$ is either of type (C) or (F), with respect to 
$$
\psi_2: \P(K)\ra L^\times/l^\times \ra L^\times/L_2^\times,
$$
and there exists a nontrivial valuation $\mu$ of $K$ such that 
 $\psi_2$ factors through $\Gamma_{\mu}$.
 \end{itemize}

\

{\em Degenerate case: there are no $(\mathrm{I})$--lines but 
there exist $(\bar{\mathrm{I}})$-lines whose
images are contained in $L_1^\times/l^\times$, for some 1-dimensional subfield of $L$.} 

\begin{itemize}
\item 
Let $L_2$ be its algebraic closure in $L$.  There may exist lines of type
(N),  $(\bar{\mathrm{N}})$,  or $(\bar{\mathrm{I}})$, but their images are contained in 
$L_2^\times/l^\times$.  
Every $\ml\subset \P(K)$ is of type (C), (F), or $(\bar{\mathrm{F}})$, with respect to
$$
\psi_2: \P(K)\ra L^\times/l^\times \ra L^\times/L_2^\times,
$$
and there exists a nontrivial valuation $\nu$ of $K$ such that 
 $\psi_2$ factors through $\Gamma_{\mu}$.
 \end{itemize}

\

Thus, in all the degenerate cases the homomorphism
$$
\psi_{\ell}: K^\times/k^\times \ra L^\times/L_2^\times,
$$
is a flag map, thus arises from a  
nontrivial valuation $\mu$,

\centerline{
\xymatrix{
1 \ar[r] & \mo_{\nu}^\times \ar[r] & K^\times \ar[r]^{\mu} \ar@{=}[d]             &                     \Gamma_{\mu} \ar[r] \ar[d]^{r} & 1 \\
            &                                    & K^\times           \ar[r]^{\psi_{\ell} \,\,\,\, }            & L^\times/L_2^\times             &     
}
}

\noindent
i.e., $\psi_{\ell}= r\circ \mu$ 
The following Lemmas will show that $\psi$ is either as in (V) or (VP) of 
Theorem~\ref{principal_theorem}.

\begin{lemm}
\label{lemm:improve}
There is a valuation $\nu$ of $K$ and a surjective homomorphism of ordered groups
$$
\Gamma_{\mu}\stackrel{\gamma}{\lra}\Gamma_{\nu}
$$
such that
\begin{enumerate}
\item $\nu = \gamma\circ \mu: K^\times\to \Gamma_{\nu}$
is a surjective map of ordered groups with $\Ker(\gamma)\subset \Ker(r)$.
\item  $\psi((1+ \mm_{\nu})^\times)= 1$.
\end{enumerate}
\end{lemm}

\begin{proof}
Let $z\in \mo_{\mu}^\times$  be such that
$r(\mu(z))\neq 0$ and thus $\psi_\ell(z)\neq 1\in L^\times/L_2^\times$.
Let $x\in \mo_{\mu}^\times \subset \Ker(\psi_\ell)$. 
We have 
$$
\mu(x+ az)=\mu(x),\mu(a)\geq 0,
$$ 
and 
$r$ is nonconstant on $\ml(z,x)$.
Thus $\psi$ is a flag map on $\ml$, and 
$$
\psi(x+ az)=\psi(x)
$$ 
so that 
$\psi(1+ az/x)= 1$. Note that $zx$ also
has $r(\mu(zx))\neq 0$ and hence we can apply the same to
$zx$, obtaining $\psi(1+ az)= 1$, for any $z$ with $r(\mu(z)) > 0$.

Note that elements $z$ with $\nu(z)= \alpha$ generate additively
the subgroup $K_{\alpha}\subset K$.
Now the elements of the form $1+z$ with $\nu(1+  z)=0$ 
generate the multiplicative
subgroup $(1 + K_{\alpha})^\times$.
Indeed, consider 
$$
(1+z)(1+z')= 1+z+z'+zz'= (1+z+z') \left(1+\frac{zz'}{1+z+z'}\right),
$$
where $\mu(z)=\mu(z')$ and $(1+z+z')\in \mo_{\mu}^\times$.
Since $\psi_l(zz')\neq 1$ we have 
$$
\psi\left(1+\frac{zz'}{1+z+z'}\right)= 1,
$$
by the same argument applied  to $z,z'$; thus
$\psi\equiv 1 $ on $(1 + K_{\alpha})^\times$.
This implies that $\psi(1+y)= 1$, even if $r(\nu(y))= 0$ but
there is a $z,r(\nu(z))\neq 1$ and $\nu(z) < \nu(y)$.
Consider the subset $\Gamma_{\mu}^+,\mu \geq 0$ in $\Gamma_{\mu}$.
Since $L^\times/L_2^\times$ is torsion-free, 
$$
\rk_{\Q}(\Ker(r))<\rk_{\Q}(\Gamma_{\mu}).
$$ 
Hence it intersects $\Gamma_{\mu}^+$ in a proper subsemigroup $\Ker(r_{\mu})^+$
and the subset of elements $s\in \Ker(r_{\mu})^+$ with $s > \mu(x)$
for any $x\in \Gamma_{\mu}^+ - \Ker(r_{\mu})^+$.

We are looking a subset of elements $S$ inside $\Ker(r_{\mu})^+ -0$
such that for each $s\in S$ such that $s < u$ for any
$u >0$ with $r(\mu(u)) \neq 0$.
Note that $S$ has to contain smallest elements in
$\Gamma_{\mu}^+ \setminus 0$ if there  are ones.
Assume that $s,s'\in S, s,s' < u,r(u)\neq 0$ and $s+s' > u$.
Note that $s + s'- u > 0$ and $s > u- s'> 0$ but
$r(u- s')\neq 0$ which provides a contradiction.
Thus $S$ is an ordered subsemigroup in $\Ker(r_{\mu})^+ -0$
which generates an ordered subgroup $\langle S\rangle $ such that 
$$
K^\times \ra \Gamma_\mu/\langle S\rangle  =:\Gamma_\nu
$$
is a valuation map for some valuation $\nu$. For this valuation, $\Ker(\nu)\supset (1+\mm_{\nu})^\times$, 
by the computation above. 
\end{proof}

\bibliographystyle{plain}
\bibliography{fieldhom}

\end{document}